\documentclass[10pt]{article}

\usepackage[utf8]{inputenc}
\usepackage{bbm}
\usepackage{relsize}
\usepackage{authblk}
\usepackage{geometry}
\usepackage{amsthm}
\usepackage{
	amsfonts,
	amsmath,
	appendix,
	color,
	amssymb,
	graphicx,
	bm,
	stmaryrd,
	mathabx,
	amssymb,
	mathdots,
	esint
}

\usepackage{hyperref}
\usepackage{xcolor}

\numberwithin{equation}{section}
\theoremstyle{plain}

\newcommand{\R}{\mathbb R}
\newcommand{\Z}{\mathbb Z}

\newcommand{\N}{\mathbb N}
\newcommand{\C}{\mathbb C}
\newcommand{\D}{\mathbb D}
\renewcommand{\P}[1]{\mathbb{P}\left[#1\right]}
\newcommand{\E}[1]{\mathbb{E}\left[#1\right]}

\DeclareMathOperator{\sech}{sech}

\newtheorem{theorem}{Theorem}[section]
\newtheorem{conjecture}[theorem]{Conjecture}
\newtheorem{proposition}[theorem]{Proposition}
\newtheorem{lemma}[theorem]{Lemma}
\newtheorem{corollary}[theorem]{Corollary}

\newtheorem{example}{Example}[section]

\newtheorem{remark}[example]{Remark}

\newcommand{\eps}{\varepsilon}
\renewcommand{\det}{{\,\rm det}\:}

\newcommand{\Pf}{{\,\rm Pf}\:}
\newcommand{\Tr}{{\,\rm Tr}\:}
\newcommand{\tr}{{\,\rm tr}\:}

\newcommand{\sgn}{{\,\rm sgn}}

\renewcommand{\Im}{{\mathrm{Im}}}
\renewcommand{\exp}[1]{\mathrm{exp} \left\{#1\right\}}

\newcommand{\de}[1]{\mathrm{d}\, #1}
\newcommand{\De}[1]{\mathrm{d}\, #1}

\newtheorem{myremarks}[example]{Remarks}

\newenvironment{remarks}{\begin{myremarks}\begin{nummer}}%
    {\end{nummer}\end{myremarks}}
    {\end{nummer}\end{myexamples}}

\newcounter{numcount}
\newcommand{\labelnummer}{(\roman{numcount})}%

\makeatletter
\providecommand{\showkeyslabelformat}[1]{\relax}        
\let\mysaveformat\showkeyslabelformat                   %
\def\myformat#1{\raisebox{-1.5ex}{\mysaveformat{#1}}}   %

\newenvironment{nummer}%
  {\let\curlabelspeicher\@currentlabel%
    \begin{list}{\textup{\labelnummer}}%
      {\usecounter{numcount}\leftmargin0pt%
        \topsep0.5ex\partopsep2ex\parsep0pt\itemsep0ex\@plus1\p@%
        \labelwidth2.5em\itemindent3.5em\labelsep1em%
      }%
    \let\saveitem\item%
    \def\item{\saveitem%
      \def\@currentlabel{\curlabelspeicher\kern.1em\labelnummer}}%
    \let\savelabel\label%
    \def\label##1{{\ifnum\thenumcount=1\let\showkeyslabelformat\myformat\fi\savelabel{##1}}%
										{\def\@currentlabel{\labelnummer}%
									 	\let\showkeyslabelformat\@gobble
									 	\savelabel{##1item}%
										}%
	   							}%
  }{\end{list}}%

  {\let\curlabelspeicher\@currentlabel%
    \begin{list}{\textup{\labelnummer}}%
      {\usecounter{numcount}\leftmargin0pt%
        \topsep0.5ex\partopsep2ex\parsep0pt\itemsep0ex\@plus1\p@%
        \labelwidth2.5em\itemindent0em\labelsep1em%
        \leftmargin2.5em}%
    \let\saveitem\item%
    \def\item{\saveitem%
      \def\@currentlabel{\curlabelspeicher\kern.1em\labelnummer}}%
    \let\savelabel\label%
    \def\label##1{{\ifnum\thenumcount=1\let\showkeyslabelformat\myformat\fi\savelabel{##1}}%
										{\def\@currentlabel{\labelnummer}%
									 	\let\showkeyslabelformat\@gobble
									 	\savelabel{##1item}%
										}%
    							}%
  }{\end{list}}%

\usepackage{enumitem}

\let\OldItem\item
\newcommand{\MyItem}[2][]{}%
\newcommand{\MR}[1]{\href{http://www.ams.org/mathscinet-getitem?mr=#1}{MR#1}}

\title{On pure complex spectrum for truncations of random orthogonal matrices and 
	Kac polynomials.}

\author[1]{Martin Gebert \thanks{mgebert@math.ucdavis.edu}}
\author[2]{Mihail Poplavskyi \thanks{mihail.poplavskyi@kcl.ac.uk}}
\affil[1]{University of California, Davis}
\affil[2]{King's College London}

\begin{document}
	\maketitle
		\begin{abstract}		\noindent	 
			Let $O(2n+\ell)$ be the group of orthogonal matrices of size 
			$\left(2n+\ell\right)\times \left(2n+\ell\right)$ 
			equipped with the probability distribution
			given by normalized Haar measure. We study the probability
			\begin{equation*}
				p_{2n}^{\left(\ell\right)} = \P{M_{2n} \, \mbox{has no real eigenvalues}},
			\end{equation*}
			where $M_{2n}$ is the $2n\times 2n$ left top minor of a $(2n+\ell)\times(2n+\ell)$ orthogonal
			matrix. We prove that this probability is given in terms of a determinant 
			identity minus a weighted
			Hankel matrix of size $n\times n$ that depends on the truncation parameter $\ell$.
			For $\ell=1$ the matrix coincides with the Hilbert matrix and we prove 
			\begin{equation*}
				p_{2n}^{\left(1\right)} \sim n^{-3/8}, \mbox{ when }n \to \infty.
			\end{equation*}
			We also discuss connections of the above to the persistence probability
			for random Kac polynomials.
		\end{abstract}

\section{Introduction}
In this paper we consider truncations of random orthogonal matrices distributed according to Haar measure. We are mostly interested in the set of real eigenvalues of these matrices and in the so-called persistence probability. This is the probability of the truncated random orthogonal matrix having no real eigenvalue. However, before we go into details regarding the results, we first want to give further motivation for considering this model in terms of random Kac polynomials.
Let $\left\{a_i\right\}_{i=0}^{\infty}$ be a sequence of statistically independent copies 
of a random variable $\xi$ with zero mean and unit variance. Then we define the random polynomial
\begin{equation}\label{e:Kac_pol}
	K_N\left(z\right) = \sum\limits_{k=0}^{N} a_kz^k
\end{equation}
having random real-valued coefficients and random roots. 
Their probability distribution and quantitative properties are a centrepiece of the theory of random polynomials
dating back to the $18^{th}$ century and attracting lots of attention since then.
It was shown in \cite{SS:62}, under mild conditions on the probability distribution of $\xi$, that
 the normalized counting measure of the zeros converges to the uniform distribution on the 
unit circle when $N\to\infty$. Restricting ourselves to real roots only,
the first major problem is determining their number
\begin{equation*}
	\mathcal{N}_{\R}(N) = \# \left\{x \in \R: K_N\left(x\right) = 0\right\},
\end{equation*}
and its asymptotic behaviour when $N\to\infty$. The problem
of calculating $\E{\mathcal{N}_{\R}(N)}$ has a long history. The first significant results were obtained in a series of papers by Littlewood \& Offord \cite{LO:38,LO:39,LO:43} 
and later improved
by Kac in \cite{Kac:43} where the author derives 
in the case of $a_i\stackrel{d}{=} \mathcal{N}(0,1)$ 
	\begin{equation*}
		\E{\mathcal{N}_{\R}(N)} = \Big(\frac{2}{\pi}+o(1)\Big)\log N, \quad N\to\infty.
	\end{equation*}
Later on this was shown to be universal for a wide class of probability distributions for $\xi$
by Kac \cite{Kac:49}, Erd\"{o}s  \& Offord \cite{EO:56}, Ibragimov 
\& Maslova \cite{IM:68,IM:71_1,IM:71_2,IM:71_3}, Tao \& Vu \cite{TV:15}, see especially 
 \cite[Sec. 1.2]{TV:15} for further references.
	
Further studies of real zeros of Kac polynomials led researchers to the calculation 
of correlation functions. In \cite{BD:97} Bleher and Di found an explicit formula
for all correlation functions between zeros which is given in terms of a multidimensional integral
of Gaussian type. Unfortunately, their answer does not allow further progress in the study of 
the distribution of the zeros and did not
show any special structure of the correlation functions. 
However, P. Forrester found in \cite{Fo:10}, generalizing methods of \cite{HKPV:09}, a Pfaffian structure 
in the seemingly different model of eigenvalues of truncated random orthogonal matrices 
when $N\to\infty$. 
He argued that Blaschke products of 
the eigenvalues of rank-one
 truncated random orthogonal matrices
converge to  an infinite series  with $\mathcal N\left(0,1\right)$-distributed coefficients. 
Here rank-one refers to truncations by one column and one row.
Comparing this to \eqref{e:Kac_pol}, it is natural to call this series
\begin{equation}\label{e:Kac_ser}
	K_{\infty}\left(z\right) = \sum\limits_{k=0}^{\infty} a_kz^k, \quad a_k \in \R
	\mbox{ are i.i.d. } N\left(0,1\right),
\end{equation}
Kac series. This manifests the connection of roots of random polynomials and eigenvalues of truncated random orthogonal matrices for large $N$.
The model of truncated random orthogonal matrices was first studied in \cite{KSZ:10} where it was shown that the eigenvalues
form a Pfaffian Point Process (PPP), meaning that the correlation functions can be expressed in terms of Pfaffians $\Pf$ of the form
\begin{equation*}
	\rho^{\left(N\right)}_k \left(z_1, z_2, \ldots, z_k\right) = 
	\Pf \left\{\mathcal{K}^{\left(N\right)}_{2\times 2}\left(z_p,z_q\right)\right\}_{p,q=1}^k,
\end{equation*}
\begin{equation*}
	\mathcal{K}^{\left(N\right)}_{2\times 2}\left(z,w\right) = 
	\begin{pmatrix}
		K^{\left(N\right)}_{11}\left(z,w\right) & K^{\left(N\right)}_{12}\left(z,w\right) \\
		- K^{\left(N\right)}_{12}\left(w,z\right) & K^{\left(N\right)}_{22}\left(z,w\right)
	\end{pmatrix},
\end{equation*}
with skew-symmetric functions $K_{11}\left(z,w\right) = - K_{11}\left(w,z\right)$ and
$K_{22}\left(z,w\right) = - K_{22}\left(w,z\right)$. For the definition and basic properties
of Pfaffians we refer the reader to Appendix~\ref{app:Auxiliary}.
Later Matsumoto and Shirai in \cite{MS:13},
without using any relation to random matrices, proved that the roots of the random Kac series \eqref{e:Kac_ser}
form a PPP as well, with corresponding kernel being just the pointwise limit of the one obtained in \cite{Fo:10}.
This was another strong evidence that there is a hidden Pfaffian structure
behind random roots of Kac polynomials. This Pfaffian structure was recently explained in \cite{PS:18}
in terms of Gaussian Stationary Process with $\sech(t/2)$ correlation. This Gaussian 
process was introduced to the area in \cite{DPSZ:02} when studying the so-called persistence probability 
for Kac polynomials. 
The persistence probability of Kac polynomials is defined by 
\begin{equation*}
	p_N := \P{\mathcal{N}_{\R}(N) = 0} = 2 \P{K_N\left(x\right) > 0, \forall x \in \R}.
	\footnote{usually the persistence probability is defined as half of what we use here} 
\end{equation*}
Obviously, this question makes no sense 
for odd values of $N$, and therefore from now and on we put $N = 2n$ for some integer $n$. One can also study
more complicated probabilities of having some prescribed number of real roots as well, i.e.
\begin{equation}\label{eq:probPN2k}
	p_{N,k}:= \P{\mathcal{N}_{\R}(N) = k}.
\end{equation}
First results on the persistence probability were obtained in \cite{LO:39}, where it is proved that $p_{2n} = 
O\left(1/\log n\right)$. Only $60$ years later power-like decay
	\begin{equation}\label{eq:probPN2k2}
		p_{2n} \sim n^{-4\theta_{\text{Kac}}},\mbox{when } n\to\infty,
	\end{equation}
for some unknown $\theta_{\text{Kac}}$,  was proven in \cite{DPSZ:02} reducing the problem to 
the study of persistence probabilities for Gaussian Stationary Processes (GSP) 
$Y_t$ with correlation function 
\begin{equation*}
	R\left(t\right) = \big\langle Y_0Y_t\big\rangle = \sech(t/2),
\end{equation*}
where $\big\langle\,\cdot\,\big\rangle$ denotes the expectation value. 
The authors showed that
	\begin{equation*}
		\P{Y_t \geq 0, \forall \, 0\leq t\leq T} \sim e^{-\theta_{\text{Kac}} T},
	\end{equation*}
with $\theta_{\text{Kac}}$ being the same as in \eqref{eq:probPN2k2}. Despite being explicitly defined, the
constant $\theta_{\text{Kac}}$ remains unknown and the best known results are 
$\theta_{\text{Kac}} \in \left(1/8, 1/4\right]$ (theoretically, \cite{LS:02}) and
$\theta_{\text{Kac}} \eqsim 0.1875\pm 0.01$ (numerically, \cite{DPSZ:02},\cite{MS:08}). 
Over time this constant became very popular as it had appeared in many applications
such as persistence of integrated Brownian motion \cite{AD:13}, \cite{Vy:10}, no flipping probabilities in Ising spin 
model \cite{DHP:96}, persistence probabilities for solutions of diffusion and heat equations with random 
initial data \cite{MS:08, MSBC:96, MS:07, DM:15}. However no single model was rigorously solved and the 
constant remained unknown. Gaussian stationary processes and the calculation of the corresponding persistence constant 
using Pfaffian structure is the main content of the upcoming paper \cite{PS:19}.

In the paper, however, using the connection suggested in \cite{Fo:10} and explained before, 
we analyse the model of truncated orthogonal 
matrices and the corresponding "persistence" probability. Let $\ell\in\N$ and $O\left(N+\ell\right)$ be the group of
orthogonal $(N+\ell)\times(N+\ell)$ matrices equipped with the probability distribution given by normalized Haar measure. 
Decomposing $O \in O\left(N+\ell\right)$ according to
\begin{equation}\label{e:O_split}
	O = 
	\left(
	\begin{array}{cc}
	M_N & B_{N\times \ell} \\
	C_{\ell \times N} & D_{\ell}
	\end{array}
	\right),
\end{equation}
the main result of the paper can be formulated as follows.

\begin{theorem}\label{t:Main_Det}
	Let $\left\{M_{2n}\right\}$ be the ensemble of the $2n\times 2n$
	top left minor of the orthogonal matrices of size $\left(2n+\ell\right)\times\left(2n+\ell\right)$
	chosen uniformly (with respect to Haar measure) at random. Then the "persistence" probability
	\begin{equation}\label{e:Pers_det}
		p^{\left(\ell\right)}_{2n}:=\P{M_{2n} \text{ has no real eigenvalues }}
		= \det\left(I_n - \mathcal{D}^{\left(\ell\right)}_n 
		H^{\left(\ell\right)}_n \mathcal{D}^{\left(\ell\right)}_n\right),
	\end{equation}
	is a determinant given in terms of the $n\times n$ Hankel matrix 
	\begin{equation*}
		\big(H^{\left(\ell\right)}_n\big)_{p,q}= 
		\frac{B\left(p+q+1/2,\ell\right)}{2^{\ell-1}\Gamma^2\left(\frac{\ell}{2}\right)},
		\quad p,q = \overline{0,n-1},
	\end{equation*}	
	where $B$ stands for the Beta-function, and the $n\times n$ diagonal matrix $\mathcal{D}^{\left(\ell\right)}$ with diagonal elements
	\begin{equation}\label{e:D_nl}
		\big(\mathcal{D}^{\left(\ell\right)}_n\big)_{p,p}
		= \sqrt{\frac{\Gamma\left(2p+\ell\right)}{\Gamma\left(2p+1\right)}},
		\quad p=\overline{0,n-1}.
	\end{equation}
\end{theorem}
\begin{remark}
	For $\ell = 1$ the above expression further simplifies to 
	\begin{equation}\label{e:Gap_Hilbert}
	p_{2n}^{\left(1\right)} = \det\big(I_n - H^{\left(1\right)}_n\big),
	\text{ where } \big(H^{\left(1\right)}_n\big)_{p,q} = \frac{1}{\pi\left(p+q+1/2\right)},
	\end{equation}
	$p,q = \overline{0,n-1}$.
	From now on we use the superscript $^{\left(\ell\right)}$ only when $\ell \neq 1$, otherwise 
	we drop it.
\end{remark}
In fact our method allows us to find not only the probability of having no real eigenvalues,
but also the moment generating function of the number of real eigenvalues.
\begin{proposition}\label{p:MGF}
	Let $\mathcal{N}^{\left(\ell\right)}_{n}\left(\R\right)$ denote a number of real 
	eigenvalues of the random matrix $M_{2n}$ taken from the ensemble defined
	above. Then for any $s \in \C$
	\begin{equation}\label{e:MGF}
		\Big\langle e^{s\mathcal{N}^{\left(\ell\right)}_n(\R)}\Big\rangle_{M_{2n}} = 
		\det\left(I_n - \left(1-e^{2s}\right)
		\mathcal{D}^{\left(\ell\right)}_n H^{\left(\ell\right)}_n \mathcal{D}^{\left(\ell\right)}_n\right),
	\end{equation}
where $\big\langle\,\cdot\,\big\rangle_{M_{2n}}$ denotes the expectation value with respect to the ensemble.
\end{proposition}
A similar result for the moment generating function of the number
of real eigenvalues for products of truncated orthogonal matrices was recently also obtained in \cite{FIK:18}.
The result was obtained in the regime of large truncations $\ell \geq N$
and is expressed in terms of Meijer G-functions, which makes its asymptotic analysis not feasible by our methods.
We are interested in the application of our result to the study of random Kac polynomials,
and therefore we stay in the so-called universality class of weak non-orthogonality (see review
\cite{FS:03} and references therein).
\begin{proposition}\label{p:all_real}
	Identity \eqref{e:MGF} gives access to the probability of all eigenvalues being real. 
	The result reads
	\begin{equation}\label{e:all_real}
	p_{2n,2n}^{\left(\ell\right)}=
	\frac{G\left(n+\frac{\ell}{2}\right)G\left(n+\frac{\ell+1}{2}\right)G\left(n+\ell\right)
		G\left(n+\ell-\frac{1}{2}\right)}{
		\Gamma^{2n}\left(\frac{\ell}{2}\right)G\left(\frac{\ell}{2}\right)
		G\left(\frac{\ell+1}{2}\right)G\left(\ell\right)G\left(2n+\ell-\frac{1}{2}\right)},
	\end{equation}
	where $G$ is the Barnes $G$-function and $p^{\left(\ell\right)}_{2n,2n}$ is defined 
	similar to \eqref{eq:probPN2k}.
	For $n\to\infty$ and $\ell$ either growing with $n$ or being fixed the probability  
	of pure real spectrum has the asymptotic expansion
	\begin{equation}\label{e:all_real_1}
		\lim\limits_{n\to\infty} \frac{\log p_{2n,2n}^{\left(\ell\right)}}{n^2} = 
		\begin{cases}
			-2\log 2, & \ell/n\to 0, \\
			\phi\left(\alpha\right), 
			& \ell/n = \alpha \in \left(0,\infty\right), \\
			-\log 2, & \ell/n \to\infty.
		\end{cases}
	\end{equation}
	where
	\begin{equation*}
		\phi\left(\alpha\right) = -\log 2 - \alpha\left(1+\frac{3}{4}\alpha\right)\log\alpha
		-\frac{\alpha }{2}
		+\left(1+\alpha\right)^2\log\left(1+\alpha\right)
		- \left(1+\frac{\alpha}{2}\right)^2\log\left(2+\alpha\right),
	\end{equation*}
	with $\phi\left(0\right) = -2\log 2$ and $\phi\left(\infty\right) = -\log 2$.

\end{proposition}
Formula \eqref{e:all_real} was previously obtained in \cite[Cor. 2]{FK:18} by a different method.
Our main result here is the asymptotics \eqref{e:all_real_1} which clearly gives an interpolation between the
weak non-orthogonality class (small $\ell$) and the Real Ginibre Ensemble ($\ell \gg n$) in terms of the
 probability $p_{2n,2n}^{\left(\ell\right)}$. The corresponding result for the Real Ginibre ensemble can be found in
\cite[Cor. 7.1]{Ed:97}.

The second main result of this paper is the asymptotic analysis of the probability
\eqref{e:Gap_Hilbert}. 
\begin{theorem}\label{t:det_main}
	The asymptotics
	\begin{equation}\label{e:asympt_det}
	\log \det\left(I_n-H_n\right) = -2\theta\log n
	+ o\left(\log n\right)
	\end{equation}
	holds as $n\to\infty$ with
	\begin{equation}\label{decayexp}
	\theta= -\frac{1}{2\pi} \int\limits_{0}^{\infty}
	\log\left(1-\sech(\pi u)\right)\text d {u}= \frac 3 {16}.
	\end{equation}
	In particular this implies that
	\begin{equation}\label{asymptProbab}
	\lim\limits_{n\to \infty}\frac{\log p_{2n}}{\log n} =-\frac 3 {8}.
	\end{equation}
	\end{theorem}
	
	\begin{remarks}
	\item The constant $3/8$ was recently also encountered in \cite{CH:17} in the context of persistence probability
	of so-called Peron polynomials. This is a very intriguing coincidence because Kac polynomials
	are defined through random coefficients, while the probability distribution on the space of Perron 
	polynomials is defined in a far more complicated way.
	\item
		The relation of truncated orthogonal matrices to Kac polynomials in the case 
		$\ell=1$ explained earlier on strongly indicates that the decay exponent 
		$4\theta_{\text{Kac}}$ in  \eqref{eq:probPN2k2} for the persistence probability 
		of Kac polynomials is given by
		\begin{equation}
		\theta_{\text{Kac}}= \frac 3 {16}.
		\end{equation}
		as well and therefore $4\theta_{\text{Kac}}= 12/16=3/4$.
		 This will be content of the upcoming paper \cite{PS:19}.
		Also note that introducing the coefficient in \eqref{e:asympt_det} as "$2\theta$"
		is justified by the following observation: All eigenvalues of the random matrix 
		$M_{2n}$ are positioned inside the unit disk and can model random roots 
		of a polynomial \eqref{e:Kac_pol} only within this domain. However, random roots of Kac 
		polynomial lying outside of the unit disk in the large $N$ limit
		are independent and equally distributed	(up to a transformation $z\to 1/z$) with those
		inside. Therefore, the persistence constant for Kac polynomials is expected to be twice as big 
		as the one for truncated orthogonal matrices.	
	\item
	Related asymptotics of the form \eqref{e:asympt_det} for $\det(I- \alpha H_n)$ 
	with $|\alpha|<1$ were proven in \cite{GF:18}. It is shown there that 
	\begin{equation*}
	\det(I- \alpha H_n)= -\frac 1 {2\pi^2} \big( \arcsin^2(\alpha)+\pi \arcsin(\alpha)\big)\log n + o(\log n). 
	\end{equation*}
	The proof works for a larger class of Hankel matrices but does not generalize to 
	$\alpha=1$ which we need in our case. This result yields that the moment generating function
	of the number of real eigenvalues can be written as
	\begin{equation*}
		\left\langle e^{s\mathcal{N}_n}\right\rangle = \left(\frac{1}{8} - \frac{2}{\pi^2}
		\left[\arccos\frac{e^s}{\sqrt{2}}\right]^2\right)
		\log n + o\left(\log n\right), \quad s\in \left(-\infty,\frac{\log2}{2}\right].
	\end{equation*}
	To conclude, we would like to mention that the same expression previously 
	appeared in \cite[Eq. (7)]{DHP:96} when studying the problem of the number of
	persistent spins in the Ising model on the half-line.
	\end{remarks}

The matrix $H_n$ is a variation of the Hilbert matrix, known since the end of the
$19^{th}$ century. Its spectral properties are well studied but, to the best of our knowledge,
the known results do not allow us to compute the determinant we are interested in.
The determinant can also be considered in the context of so-called
Toeplitz\,$\pm$\,Hankel determinants. These determinants are of independent interest where 
we refer to the review \cite{DIK:13} and to
\cite{DIK:11, BE:01, BE:02, BE:09, BE:17} for recent progress in that area. The results in the
 mentioned papers solely deal with Toeplitz\,+\,Hankel determinants of particular forms,
more precisely of Toeplitz and Hankel matrices having either the same symbol or symbols which differ by
just a factor $e^{\pm it}$. Our result corresponds to symbols
\begin{equation*}
	\sigma^T\left(e^{it}\right) \equiv 1, \quad 
	\sigma^H\left(e^{it}\right) = i e^{-it/2},
\end{equation*}
where the last one has a jump discontinuity at $t=0$.

The paper is organised as follows. In Section~\ref{s:RMT} we prove
Theorem~\ref{t:Main_Det} and derive \eqref{e:MGF}~-~\eqref{e:all_real_1}. 
The proof of Theorem~\ref{t:det_main} is the main
content of Section~\ref{s:Det}. The proofs of a number of auxilliary results is deferred to Sections~\ref {pf:lemma1}--\ref{pf:lemma333}.
Finally, in Section~\ref{s:OpenProblems} we discuss some open problems and conjectures in the scope of our 
interests.

\textbf{Acknowledgement:} The authors would like to thank Gregory Schehr 
for bringing the problem to their attention and stimulating this research
by providing links to other yet unsolved problems. We
are very grateful to Emilio Fedele and Alexander Pushnitski for numerous 
fruitful discussions on modern progress of 
Hankel matrices and sharing their unpublished results
together with the correction to the result of \cite[Thm. 4.3]{Wi:66} 
(see Lemma~\ref{l:moments}). This research was partially supported by ERC starting 
grant SPECTRUM (639305) (MG) and by EPRSC Grant No. EP/N009436/1 (MP).

\section{Ensemble of truncated orthogonal matrices}\label{s:RMT}
In this section we give the details for the derivation of \eqref{e:Pers_det}. 
First we find the joint probability distribution
function for the eigenvalues of $M_{2n}$. It was previously calculated with a mistake in the coefficients in
\cite{KSZ:10} and then corrected by a different method in \cite{Mays:11} in the case of "large" truncations.
For rank-one truncations ($\ell=1$) this was also calculated in \cite[Thm. 6.4, Rmk. 2]{KK:17}.
We follow the strategy of \cite{KSZ:10}, to obtain the result for small truncations,
and fill some holes and correct some minor mistakes in their proof. In the second part we note that the 
distribution fits in the framework of Point Processes Associated to Weights developed in \cite{Si:07, BS:09}.
This enables us to calculate the corresponding averages in terms of a family 
of skew-orthogonal polynomials found in \cite{Fo:10}.

\begin{proof}[Proof of Theorem~\ref{t:Main_Det}]
We start by considering the orthogonal group of size $N+\ell$ with the probability distribution
defined by the measure 
\begin{equation}\label{e:measure_O}
	\de{\mu\left(O\right)} = \frac{1}{v_{N+\ell}}\delta\left(O^TO-I_{N+\ell}\right)\De{O},
\end{equation}
where $\De{O} = \prod\limits_{i,j=1}^{N+\ell} O_{i,j}$ is the flat Lebesgue measure on $\R^{\left(N+\ell\right)^2}$
and 
\begin{equation}\label{e:measure_OO}
	v_{N+\ell} = \int_{\R^{(N+\ell)^2}} \delta\left(O^TO-I_{N+\ell}\right)\De{O}= \prod\limits_{j=1}^{N+\ell} 
	\frac{\pi^{j/2}}{\Gamma\left(j/2\right)},
\end{equation}	
is the volume of the orthogonal group
(see Proposition~\ref{p:v_N}). We decompose $O\in O\left(N+\ell\right)$ as (cf.
\eqref{e:O_split})
\begin{equation*}
	O = 
	\left(
	\begin{array}{cc}
	M_N & B_{N\times \ell} \\
	C_{\ell \times N} & D_{\ell}
	\end{array}
	\right).
\end{equation*}
Firstly, we compute the induced measure 
for the ensemble of the top-left minor $M_N$. We denote the space of all possible 
matrices $M_N$ by $\mathcal{O}^{\left(\ell\right)}_N$. 
One can integrate out (see Lemma~\ref{l:MC_distr} for details) the variables
$B_{N\times\ell}$ and $D_{\ell}$ to obtain the joint distribution of $M_N$
and $C_{N\times \ell}$. Then the probability distribution on 
$\mathcal{O}_N^{\left(\ell\right)}$ is written as
\begin{equation}\label{e:denMN}
\de{P\left(M_N\right)} =\frac{v_{\ell}}{v_{N+\ell}}  \Big( \int_{\R^{\ell N}}
\delta \left(M_N^TM_N+ C_{\ell \times N}^T C_{\ell \times N} - I_N\right)
\De{C_{\ell \times N}}\Big)\De{M_N},
\end{equation}	
where $\De{M_N}$ is Lebesgue measure on $\R^{N^2}$ and $\De{C_{\ell \times N}}$ is the Lebesgue measure on $\R^{N\ell}$. The constraint imposed by the $\delta$-function yields 
$M_N^TM_N = I_N - C^T_{\ell\times N}C_{\ell\times N}$, which implies that all eigenvalues 
of $M_N$ belong to the unit disk $\D=\left\{z\rvert \left|z\right|\leq 1\right\}$.
We are now left with the integration over $C_{\ell\times N}$. The dimension 
of the $\delta$-function inside the integrand is $\frac{N\left(N+1\right)}{2}$, i.e. the number of constraints imposed on the matrix in the integrand of the $\delta$-function,
and the number of independent variables in $C_{\ell \times N}^TC_{\ell \times N}$ 
is equal to $\frac{\ell \left(\ell+1\right)}{2}$. This means that in the case of $\ell\leq N$ the
integration over $C_{\ell\times N}$ will lead to a singular measure concentrated on the boundary
of the matrix ball $M_N^TM_N \leq I_N$ (see \cite{KSZ:10}). We concentrate on the case of fixed 
$\ell$ and $N \to \infty$, so the measure will be singular. If $\ell \geq N$ then the integration 
over $C_{\ell\times N}$ gives \cite{KSZ:10}
\begin{equation*}
	\de{P\left(M_N\right)} = \frac{v_{\ell}^2}{v_{N+\ell}v_{\ell-N}}
	\det\left(I-M_N^TM_N\right)_+^{\frac{\ell-N-1}{2}}\de{M_N},
\end{equation*}
where we write $M_+$ to denote $M$, when $M$ is positive definite and $0$ otherwise.

In the singular case, i.e. $\ell \leq N$, we calculate the eigenvalue distribution by following
\cite{KSZ:10} and fill missing details. 
We start
by noting that there are two types of eigenvalues for matrix $M_N$: they can be either real or come in pairs
of complex conjugate ones. We introduce disjoint subsets of $\mathcal{O}^{\ell}_N$
given by (see \cite{BS:09})
\begin{equation*}
\mathcal{X}_{L,M} = \left\{ \left. M_N \in \mathcal{O}^{\ell}_N \right| 
M_N \text{ has } L \text{ real and } M \text{
	pairs of complex eiv's}\right\},
\end{equation*}
for all pairs $L,M \in \Z_+$ such that $L+2M = N$. For all $M_{N} \in \mathcal{O}^{\ell}_N$
we define a lexicographical order of the $2M+L$ eigenvalues: 
\begin{equation*}
	\lambda_1 , \ldots, \lambda_L, x_1+iy_1,x_1-iy_1,\ldots,x_M+iy_M, x_M-iy_M,
\end{equation*}
where $\lambda_i, x_i,y_i\in \R$ and
\begin{equation*}
	\lambda_i\geq \lambda_{i+1},\,
	x_i > x_{i+1} \mbox{ or } x_i = x_{i+1} \mbox{ with } y_i>y_{i+1}.
\end{equation*}
In what follows, we restrict ourselves 
to even $N$ and therefore $L$ is even as well. In the case of odd $N$ we can obtain 
similar results, but the expressions would become lengthy and we omit this case.

It was suggested by Borodin 
and Sinclair in their seminal study \cite{BS:09} on the Real Ginibre ensemble that, instead of the full distribution, one should rather 
compute the distribution of the eigenvalues conditioned on
$M_N \in \mathcal{X}_{L,M}$. We follow this idea and prove
\begin{lemma}\label{l:CJPD}
	Let $M_N$ be the top left corner of an orthogonal matrix $O\in O(N+\ell)$ drawn randomly
	with respect to the measure \eqref{e:measure_O}. Then the density of the
	joint distribution of the ordered 
	eigenvalues of $M_N$ conditioned on having $L$ real eigenvalues 
	$\vec{\lambda}=\left(\lambda_1,\ldots,\lambda_L\right)$ and $M$ pairs of complex conjugate 
	eigenvalues $\vec{Z}=\left(z_1, \bar{z}_1,\ldots,z_M,\bar{z}_M\right)$ is
	\begin{equation*}
	p^{\left(L,M\right)}_{\ell}\left(\vec{\lambda},\vec{Z}\right) = 
	2^M \frac{v_{\ell}v_{N}}{v_{N+\ell}} \left(\frac{\left(2\pi\right)^{\ell}}{\ell!}\right)^{N/2}
	\left|\Delta\left(\vec{\lambda}\cup\vec{Z}\right)\right|
	\prod\limits_{j=1}^L w_{\ell}\left(\lambda_j\right)
	\prod\limits_{j=1}^{M} w_{\ell}\left(z_j\right)w_{\ell}\left(\bar z_j\right),
	\end{equation*}
	where $\Delta\left(\zeta_1,\ldots,\zeta_p\right) = \prod\limits_{1\leq i,j \leq p} 
	\left(\zeta_i-\zeta_j\right)$ is the Vandermonde determinant   and 
	$w_{\ell}\left(z\right) \equiv 0$ outside the unit disk $\D$ while for $z\in \D$
	\begin{equation}\label{def:omega:ell}		
	w^2_{\ell}\left(z\right) = 
	\begin{cases}
	\displaystyle\frac{1}{2\pi}\left|1-z\right|^{-1},&\ell=1,\\
	\displaystyle\frac{\ell\left(\ell-1\right)}{2\pi}
	\left|1-z^2\right|^{\ell-2}\int\limits_{
		\frac{2\left|\Im z\right|}{\left|1-z^2\right|}}^{1} \left(1-u^2\right)^{\frac{\ell-3}{2}}\de{u},
	&\ell\geq 2.
	\end{cases}
	\end{equation}
\end{lemma}
We prove the above lemma in Section \ref{pf:lemma1} that generalizes the corresponding result of \cite{Mays:11}
to any integer $\ell$ either larger than $N$ or not.
In the following we will use 
$\omega_\ell$ a lot, which is of course the square root of the latter function. 
The above expression for the conditional distribution of the eigenvalues fits in the framework of
point processes associated to weights developed by Borodin and Sinclair in \cite{BS:09}. 
We use their main result which can be stated as
\begin{theorem}\label{t:Pfaff}\cite[Thm~3.4]{BS:09}
	Let $\mathcal{P}$ be a point process of even size $N$ in the complex plane 
	containing either real or pairs of complex conjugate points. 
	Assume that the probability distribution function for ordered configurations 
	$\Xi = \left(\xi_1,\xi_2,\ldots,\xi_N\right)$ conditioned 
	on $\Xi \in \mathcal{X}_{L,M}$ is given by
	\begin{equation*}
	p^{\left(L,M\right)}\left(\Xi\right) = C_N
	2^M \left|\Delta\left(\Xi\right)\right|
	\prod\limits_{j=1}^N w\left(\xi_j\right)\mathbbm{1}_{\Xi \in X_{L,M}}
	\end{equation*}
	with $C_N$ being a normalization constant.

	Then for any sequence of monic polynomials 
	$\big\{p_k\left(z\right)=z^k+\ldots\big\}_{k=0}^{N-1}$ we obtain
	\begin{enumerate}
		\item For any function $f:\C\to\C$
		\begin{equation}\label{e:average}
		\Big\langle \prod_{j=1}^N f\left(\xi_j\right)\Big\rangle_{\mathcal{P}}=
		\frac{\Pf\left(U_{i,j}\left(f\right)\right)}{\Pf\left(U_{i,j}\left(1\right)\right)},
		\end{equation}
		where $\left\langle\,\cdot\,\right\rangle_{\mathcal{P}}$ is the average with respect to $\mathcal P$ and
		\begin{equation*}
			U_{i,j}\left(f\right) 
			=
			\int\limits_{\C}f\left(z\right)p_{i}\left(z\right)w\left(z\right)
			\left(\varepsilon \left[fp_{j}w\right]\right)\left(z\right)
			-f\left(z\right)p_{j}\left(z\right)w\left(z\right)
			\left(\varepsilon \left[fp_{i}w\right]\right)\left(z\right)
			\de{^2z}.
		\end{equation*}
		Here $\Pf$ denotes the Pfaffian of a matrix and
		$\varepsilon : S\left(\C\right)\to S\left(\C\right)$ 
		is the operator defined on the class of Schwartz functions $S(\C)$ by
		\begin{equation*}
		\left(\varepsilon \left[f\right]\right)\left(z\right) = 
		\begin{cases}
		\displaystyle\frac{1}{2}\int\limits_{\R}f\left(t\right)\sgn\left(t-z\right)\de{t}, & z\in \R, \\
		i\sgn\left(\Im z\right)f\left(\overline{z}\right),& z\in \C\setminus\R.
		\end{cases}
		\end{equation*}
		\item The point process $\mathcal{P}$ is a Pfaffian point process with kernel
		\begin{equation*}
		\mathbf{K}_N\left(z,w\right) = 
		\begin{pmatrix}
		DS_N\left(z,w\right) & S_N\left(z,w\right)\\
		-S_N\left(w,z\right) & IS_N\left(z,w\right)+\varepsilon\left(z,w\right)
		\end{pmatrix},
		\end{equation*}
		where
		\begin{align*}
			DS_N\left(z,w\right) &= 2\sum\limits_{i,j=0}^{N-1} 
			v_{i,j}p_{i}\left(z\right)p_{j}\left(w\right)w\left(z\right)w\left(w\right),\\
			S_N\left(z,w\right) &= 2\sum\limits_{i,j=0}^{N-1} 
			v_{i,j}p_{i}\left(z\right)w\left(z\right)
			\left(\varepsilon\left[p_{j}w\right]\right)\left(w\right),\\
			IS_N\left(z,w\right) &= 2\sum\limits_{i,j=0}^{N-1} 
			v_{i,j}\left(\varepsilon\left[p_{i}w\right]\right)\left(z\right)
			\left(\varepsilon\left[p_{j}w\right]\right)\left(w\right),
		\end{align*}
		\begin{equation*}
			\varepsilon\left(z,w\right) = 
			\begin{cases}
			\displaystyle\frac{1}{2}\sgn\left(z-w\right),& z,w\in \R, \\
			0,&\mbox{otherwise}.
			\end{cases}
		\end{equation*}
		and $v_{i,j}$ are the matrix elements of $\big(U_{i,j}^{-1}\left(1\right)\big)^T$.
	\end{enumerate}
\end{theorem}

\begin{remark}
	The matrix elements $U_{i,j}\left(f\right)$ can be written in terms of the skew-product 
	\begin{equation}\label{e:inner_pr}
		\left( f,g\right)^{\left(w\right)} = 
		\int\limits_{\C}f\left(z\right)w\left(z\right)
		\left(\varepsilon \left[g w\right]\right)\left(z\right)
		-g\left(z\right)w\left(z\right)
		\left(\varepsilon \left[f w\right]\right)\left(z\right)
		\de{^2z}.
	\end{equation}
	The product should be thought of consisting of two parts because	
	the action of the operator $\varepsilon$ heavily depends on the argument 
	being real or complex. More precisely one can introduce "real" and "complex" parts of the skew-product
	\begin{align*}
		\left( f,g\right)_{\R}^{\left(w\right)} &= 
		\int\limits_{\R}f\left(x\right)w\left(x\right)
		\left(\varepsilon \left[g w\right]\right)\left(x\right)
		-g\left(x\right)w\left(x\right)
		\left(\varepsilon \left[f w\right]\right)\left(x\right)
		\de{x}, \notag\\
		\left( f,g\right)_{\C}^{\left(w\right)}
		&= i\int\limits_{\R^2}\left(f\left(x+iy\right)g\left(x-iy\right)-g\left(x+iy\right)f\left(x-iy\right)\right)
		w\left(x+iy\right)w\left(x-iy\right)\sgn\left(y\right)\de{x}\de{y}
		\notag.
	\end{align*}
\end{remark}

We see that the distribution found in Lemma \ref{l:CJPD}, satisfies the assumptions of the latter theorem with 

\begin{equation*}
	C_N = \frac{v_{\ell}v_{N}}{v_{N+\ell}} \left(\frac{\left(2\pi\right)^{\ell}}{\ell!}\right)^{N/2},
	\quad 
	\omega= \omega_\ell,
\end{equation*}
defined in \eqref{def:omega:ell}. The latter formulas can be further simplified by choosing a 
family of monic polynomials skew-orthogonal 
with respect to the skew-product \eqref{e:inner_pr} with $w=w_{\ell}$, for which we use notation
$\left( f,g\right)^{\left(\ell\right)}$ (and  $\left( f,g\right)^{\left(\ell\right)}_{\R}, 
\left( f,g\right)^{\left(\ell\right)}_{\C}$ for its parts).
Such a family of skew-orthogonal polynomials
was found in \cite{Fo:10}
and later generalized in \cite{FIK:18} to the case of products of truncated orthogonal matrices. This family reads
\begin{lemma}\label{l:inner_pr}
	Let $k\in\N_0$. Then the polynomials defined by
	\begin{equation}\label{e:Skew_pols}
	\pi_{2k}\left(z\right) = z^{2k},\qquad 
	\pi_{2k+1}\left(z\right) = z^{2k+1}-\frac{2k}{2k+\ell}z^{2k-1},
	\end{equation}
	are skew-orthogonal with respect to the skew-product \eqref{e:inner_pr}, i.e. for $i,j\in\N_0$ with
	$i<j$
	\begin{equation*}
	\left( \pi_{i},\pi_{j} \right)^{\left(\ell\right)} = 
	\begin{cases}
	\displaystyle\frac{\ell!\left(2k\right)!}{\left(2k+\ell\right)!},& i=2k, j=2k+1,\\
	0,& \mbox{otherwise}.
	\end{cases}
	\end{equation*}
	Moreover for any $i,j\in\N_0$
	\begin{equation*}
		\left( \pi_{i},\pi_{j} \right)^{\left(\ell\right)}_{\R} = 
		\begin{cases}
		\displaystyle\frac{\ell!}{2^{\ell-1}\Gamma^2\left(\frac{\ell}{2}\right)\left(2q+\ell\right)}	
		B\left(p+q+\frac{1}{2},\ell\right), & i=2p, j=2q+1\\
		0, &  i-j\  \text{even}.
		\end{cases}
	\end{equation*}
\end{lemma}
We prove the above lemma in Section \ref{pf:lemma2}.
Our final step is to apply the result of \eqref{e:average} to the function
\begin{equation}\label{e:fz}
	f\left(\zeta\right) = 1-\chi_{\R}\left(\zeta\right) = 
	\begin{cases}
		1, & \zeta \in \C\setminus\R,\\
		0, & \zeta \in \R.
	\end{cases}
\end{equation}
Then the average $\left\langle \prod_{j=1}^N f\left(\xi_j\right)\right\rangle_{M_{2n}}$ 
with respect to the random matrices ensemble coincides with $p_{2n}^{\left(\ell\right)}$ and using
\eqref{e:average} we obtain
\begin{equation}\label{e:Pfaffian}
	p_{2n}^{\left(\ell\right)} = 
	\frac{\Pf \left\{\left( \pi_i,\pi_j\right)_{\C}^{\left(\ell\right)}\right\}_{i,j=0}^{2n-1}}{
		\Pf \left\{\left( \pi_i,\pi_j\right)^{\left(\ell\right)}\right\}_{i,j=0}^{2n-1}}
	= \frac{\Pf \left\{\left( \pi_i,\pi_j\right)^{\left(\ell\right)}
		-\left( \pi_i,\pi_j\right)_{\R}^{\left(\ell\right)}\right\}_{i,j=0}^{2n-1}}{
		\Pf \left\{\left( \pi_i,\pi_j\right)^{\left(\ell\right)}\right\}_{i,j=0}^{n-1}}.
\end{equation}
The matrix in the denominator is block-diagonal and the one in the numerator has a check board pattern. Both Pfaffians can
be reduced to determinants by the use of
\begin{proposition}\label{p:checkboard}
	Let $A = \left\{a_{i,j}\right\}_{i,j=0}^{2n-1}$ be a skew-symmetric matrix with $a_{i,j} = 0$
	whenever $i$ and $j$ have the same parity. Then
	\begin{equation*}
	\Pf A = \det \big\{a_{2i,2j+1}\big\}_{i,j=0}^{n-1}.
	\end{equation*}
\end{proposition}
We prove this proposition in Appendix \ref{app:Auxiliary}.
Inserting this in \eqref{e:Pfaffian}, implies
\begin{equation}\label{e:pnl}
	p_{2n}^{\left(\ell\right)} = \frac{\det \left\{\left( \pi_{2p},\pi_{2q+1}\right)^{\left(\ell\right)}
	- \left( \pi_{2p},\pi_{2q+1}\right)_{\R}^{\left(\ell\right)}\right\}_{p,q=0}^{n-1}}{\det 
	\left\{\left( \pi_{2p},\pi_{2q+1}\right)^{\left(\ell\right)}\right\}_{p,q=0}^{n-1}}.
\end{equation} 
The multiplicative property of the determinant, Lemma~\ref{l:inner_pr} and the diagonal structure 
of the matrix in the denominator yields
\begin{equation*}
	p_{2n}^{\left(\ell\right)} = \det\left\{I-\frac{\left( \pi_{2p},\pi_{2q+1}\right)_{\R}^{\left(\ell\right)}}{
	\left( \pi_{2p},\pi_{2q+1}\right)^{\left(\ell\right)}}\right\}_{p,q=0}^{n-1}
	= \det\left\{I - \frac{B\left(p+q+\frac{1}{2},\ell\right)}{2^{\ell-1}\Gamma^2\left(\frac{\ell}{2}\right)}
	\frac{\left(2p+\ell\right)!}{\left(2p\right)!\left(2q+\ell\right)}\right\}_{p,q=0}^{n-1}.
\end{equation*}
To obtain \eqref{e:Pers_det}, we conjugate the latter matrix by the diagonal matrix $Q$ given by $Q_{p,p}=  \sqrt{\frac{\left(2p\right)!}{\left(2p+\ell\right)!\left(2p+\ell\right)}}$, i.e. multiply by $Q$ from the left and $Q^{-1}$ from the right. This proves Theorem~\ref{t:Main_Det}.
\end{proof}
\begin{proof}[Proof of Proposition~\ref{p:MGF}]
	Identity \eqref{e:MGF} follows along the same lines as \eqref{e:Pers_det}. We change
	the test function \eqref{e:fz} to
	\begin{equation*}
	g\left(z\right) = 1-\left(1-e^s\right)\chi_{\R}\left(z\right).
	\end{equation*}
	One can easily see that $g\left(z\right) = e^s$ if $z\in\R$ and $1$ otherwise. Therefore,
	\begin{equation*}
	\Big\langle \prod\limits_{z\in \mathrm{spec}M_{2n}} g\left(z\right)\Big\rangle_{M_{2n}}
	= \left\langle e^{s\mathcal{N}_n^{\left(\ell\right)}\left(\R\right)}\right\rangle_{M_{2n}}.
	\end{equation*}
	Applying \eqref{e:average} to the above function, we obtain, compare also with \eqref{e:pnl}, that
	\begin{equation*}
		\left\langle e^{s\mathcal{N}^{\left(\ell\right)}_n\left(\R\right)}\right\rangle_{M_{2n}} 
		= 
		\frac{\Pf \left\{\left( \pi_i,\pi_j\right)^{\left(\ell\right)}
			-\left(1-e^{2s}\right)\left( \pi_i,\pi_j\right)_{\R}^{\left(\ell\right)}\right\}_{i,j=0}^{2n-1}}{
			\Pf \left\{\left( \pi_i,\pi_j\right)^{\left(\ell\right)}\right\}_{i,j=0}^{n-1}}.
	\end{equation*}
	Mimicking the analysis at the end of the proof of Theorem~\ref{t:Main_Det}, gives the result.
\end{proof}

\begin{proof}[Proof of Proposition~\ref{p:all_real}]
Expanding the expectation according to 
	\begin{equation}\label{c:all_real_eq1}
	\left\langle e^{s\mathcal{N}^{\left(\ell\right)}_n}\right\rangle_{M_{2n}} = \sum_{k=0}^{2n} e^{s k} \mathbb P \big(\mathcal{N}_n^{\left(\ell\right)} = 		k\big),
	\end{equation}
	implies that the probability $p^{\left(\ell\right)}_{2n,2n}$ is given by the coefficient in front of $e^{2ns}$ in this expansion. The latter expectation can be expressed in terms of the determinant	\eqref{e:MGF}. Expanding this determinant in the same way as \eqref{c:all_real_eq1} in terms of powers of $e^{sk}$ using Leibniz formula, we obtain 
	\begin{equation*}
	p^{\left(\ell\right)}_{2n,2n} = \mathrm{coeff}\left[
	\det\left(I_n-\mathcal{D}^{\left(\ell\right)}_n H^{\left(\ell\right)}_n
	\mathcal{D}^{\left(\ell\right)}_n+e^{2s}
	\mathcal{D}^{\left(\ell\right)}_n H^{\left(\ell\right)}_n
	\mathcal{D}^{\left(\ell\right)}_n\right),
	e^{2sn}
	\right],
	\end{equation*}
	where $\text{coeff}[\cdot,e^{2sn}]$ stands here for the coefficient in front of $e^{2sn}$ in this expansion.
	The power $e^{2sn}$ is the maximum power of the exponent in the expansion and thus writing down the latter determinant using the 	Leibniz rule, one sees that
	\begin{equation*}
	p_{2n,2n}^{\left(\ell\right)} = \det \mathcal{D}^{\left(\ell\right)}_n H^{\left(\ell\right)}_n
	\mathcal{D}^{\left(\ell\right)}_n = 
	\det H^{\left(\ell\right)}_n\cdot\big(\det \mathcal{D}^{\left(\ell\right)}_n\big)^2.
	\end{equation*}
	For the determinant of $H^{\left(\ell\right)}_n$ we use  a well-known identity for Hankel 
	matrix determinants. Let $\mathcal{H} = \left\{h_{j+k}\right\}_{j,k=0}^{n-1}$ be a Hankel matrix with
	\begin{equation}
	h_{n} = \int x^n \mu\left(x\right)\de{x},
	\end{equation}
	for some weight function $\mu\left(x\right)$. Then using the symmetry of the integrand
	\begin{align*}
		\det \mathcal{H} &=
		\det\left\{\int x_j^{j+k}\mu\left(x_j\right)\de{x_j}\right\}_{j,k=0}^{n-1}
		= \iint \prod\limits_{j=0}^{n-1} x_j^j \det\left\{ x_j^{k}\right\}_{j,k=0}^{n-1}
		\prod_{j=0}^{n-1} \mu\left(x_j\right)\de{x_j}
		\\
		&=
		\frac{1}{n!}\iint \sum\limits_{\sigma}
		\prod\limits_{j=0}^{n-1} x_{\sigma\left(j\right)}^j 
		\det\left\{ x_{\sigma\left(j\right)}^{k}\right\}_{j,k=0}^{n-1}
		\prod_{j=0}^{n-1} \mu\left(x_j\right)\de{x_j}
		\\
		&=
		\frac{1}{n!}\iint \prod\limits_{1\leq j<k\leq n}\left(x_j-x_k\right)^2
		\prod_{j=0}^{n-1} \mu\left(x_j\right)\de{x_j}.
	\end{align*}
	In the case of $H^{\left(\ell\right)}_n$ we take $\mu\left(x\right) = \displaystyle\frac{1}{2^{\ell-1}
		\Gamma^2\left(\frac{\ell}{2}\right)}
	x^{-1/2}\left(1-x\right)^{\ell-1}, x\in\left[0,1\right]$. Then
	\begin{align*}
		\det H^{\left(\ell\right)}_n &= \frac{1}{2^{n\left(\ell-1\right)}\Gamma^{2n}\left(\frac{\ell}{2}\right)n!}
		\prod\limits_{j=0}^{n-1}
		\frac{\Gamma\left(j+\frac{1}{2}\right)\Gamma\left(j+\ell\right)\Gamma\left(j+2\right)}{
			\Gamma\left(j+n+\ell+\frac{1}{2}\right)}\\
		&=
		\frac{1}{2^{n\left(\ell-1\right)}\Gamma^{2n}\left(\frac{\ell}{2}\right)}
		\prod\limits_{j=0}^{n-1}
		\frac{\Gamma\left(j+\ell\right)\Gamma\left(j+\frac{1}{2}\right)\Gamma\left(j+1\right)}{
			\Gamma\left(j+n+\ell-\frac{1}{2}\right)}
		,
	\end{align*}
	where we used the value of Selberg's integral \cite{Se:44}. Using definition \eqref{e:D_nl} 
	of $\mathcal{D}^{\left(\ell\right)}_n$ and the value of its determinant, 
	together with the Gamma function duplication formula, we get	
	\begin{equation*}
		p_{2n,2n}^{\left(\ell\right)} = \Gamma^{-2n}\left(\frac{\ell}{2}\right)
		\prod\limits_{j=0}^{n-1} \frac{\Gamma\left(j+\frac{\ell}{2}\right)
		\Gamma\left(j+\frac{\ell+1}{2}\right)\Gamma\left(j+\ell\right)}{
		\Gamma\left(j+n+\ell-\frac{1}{2}\right)}.
	\end{equation*} 
	The former can be rewritten as in \eqref{e:all_real} by using
	\begin{equation*}
		\prod\limits_{j=0}^{n-1}\Gamma\left(j+a\right) = \frac{G\left(n+a\right)}{G\left(a\right)},
	\end{equation*}
	with $G$ being the Barnes $G$-function.
	The asymptotic expansion of Barnes $G$-function at infinity reads
	\begin{equation*}
		\log G\left(z\right) = \frac{1}{2}z^2\log z-\frac{3}{4}z^2 -z\log z+z\left(1+\frac{\log 2\pi}{2}\right)
		+\frac{5}{12}\log z+O\left(1\right), \quad z\to\infty.
	\end{equation*}
	Let $\alpha_n = \frac{\ell}{n}$ be bounded by strictly positive constants 
	from above and below when $n\to\infty$. 
	Applying the latter asymptotic expansion to \eqref{e:all_real} and $\alpha_n$,
	we only keep the first two terms as the rest will contribute to lower order terms. We then obtain
	\begin{align*}
		\log p_{2n,2n}^{\left(\ell\right)} = n^2
		&\left(
			-\log 2 - \alpha_n\log\alpha_n\left(1+\frac{3}{4}\alpha_n\right)
			-\frac{\alpha_n }{2}
			\right.
			\\
			&\left.
			+\left(1+\alpha_n\right)^2\log\left(1+\alpha_n\right)
			- \left(1+\frac{\alpha_n}{2}\right)^2\log\left(2+\alpha_n\right)
		\right)+o\left(n^2\right), n\to\infty,
	\end{align*}
	where we also used
	\begin{equation*}
		\log \Gamma\left(z\right) = \left(z-\frac{1}{2}\right)\log z - z + o\left(z\right), z\to\infty,
	\end{equation*}
	to analyse the denominator of \eqref{e:all_real}.
	The result derived above is valid even for $\alpha_n = o\left(1\right), n\to\infty$, and thus
	for $\ell\ll n$ we get
	\begin{equation*}
		\log p_{2n,2n}^{\left(\ell\right)} = -2\log 2 n^2 + o\left(n^2\right).
	\end{equation*}
	One can also argue that for $\alpha_n\to\infty$ the coefficient in front of $n^2$ converges to
	$-\log 2$, which gives the answer. However, to prove rigorously the asymptotic result in the regime of 
	large truncations with $\ell \gg n$ we need to take all non-constant terms of the
	Barnes $G$-function expansion into account. While expanding \eqref{e:all_real} with respect to $\ell$
	one can see that terms containing $\ell^2\log \ell, \ell \log\ell$ and $\log\ell$ vanish and
	for $\beta_n = n/\ell = o\left(1\right), n\to\infty$ we get
	\begin{align*}
		\log p_{2n,2n}^{\left(\ell\right)}
		&= \bigg(\left(1+\beta_n\right)^2\log\left(1+\beta_n\right)-
		\Big(\beta_n+\frac{1}{2}\Big)^2\log\left(1+2\beta_n\right)-\frac{\beta_n}{2}
		-\beta_n^2\log 2\bigg)\ell^2
		\\&
		+\left(-\frac{5}{2}\left(1+\beta_n\right)\log\left(1+\beta_n\right)
		+\frac{3}{4}\left(1+2\beta_n\right)\log\left(1+2\beta_n\right)
		+\beta_n\left(1+\log2^{3/2}\pi\right)\right)\ell
		\\
		&+\frac{7}{24}\left(2\log\left(1+2\beta_n\right)-5\log\left(1+\beta_n\right)
		\right)+o\left(1\right), \quad
		\ell\to\infty.
	\end{align*}
	For small $\beta_n$ the first bracket is dominated by $\beta_n^2\log 2$, this together with the definition
	of $\beta_n$ gives the result. The second and third brackets are of order $O\left(\beta_n\right)$ and
	therefore will not contribute to the leading order. This finishes the proof of \eqref{e:all_real_1}.
\end{proof}
\section{Asymptotic analysis of $\det\big(I_n-H^{(1)}_n\big)$}\label{s:Det}
In this section we use the short hand notation $H_n=H^{(1)}_n$. 
The asymptotic analysis of the determinant $\det\left(I_n-H_n\right)$ follows along similar ideas used in \cite{KPTTZ:16},
where a similar persistence problem led to an analysis of a determinant of 
identity minus a weighted Hankel matrix as well.
We use the trace-log expansion of the determinant which leads to analysing traces of powers of the Hilbert matrix $H^{(1)}_n$.
The asymptotic behaviour of the trace of a fixed power of the Hilbert matrix are studied in \cite{Wi:66}. However, to obtain the asymptotic behaviour of the infinite series, we need to work a bit harder to obtain proper upper 
and a lower bound. The upper bound is simple and follows from truncating the trace-log expansion 
to a finite number of terms. For the lower bound we change $H^m_n\to H_nH_{\infty}^{m-1}$, where $H_\infty$ denotes the infinite Hilbert matrix, as it was 
suggested in \cite{KPTTZ:16}. Finally, using the explicit 
diagonalization of $H_{\infty}$ found in \cite{KS:16}, we conclude the proof.
\begin{proof}[Proof of Theorem~\ref{t:det_main}]
	In what follows we use the notation $D_n = \det\left(I_n-H_n\right)$ and
	$H=H_{\infty}$ for the infinite Hilbert matrix acting on $\ell_2\left(\N_0\right)$.
	From \cite[Sec. 4]{KS:16} it follows that $\sigma(H)=[0,1]$ and consists of purely 
	absolutely continuous spectrum only. In particular, $\|H\|=1$, where $\|\cdot\|$ 
	denotes the operator norm. For the $n\times n$ restriction we have $H_n$ of $H$ we have
		
	\begin{lemma}\label{l:Largest_eig}
		For all $n\in\N$ we obtain	$\| H_n\|  <1$.
	\end{lemma}
We prove this lemma in Subsection~\ref{pf:lemma3}.
Our analysis relies on the well-known series expansion 
	\begin{equation}\label{e:Log_series}
	\log(1-x) =-\displaystyle \sum_{m\in \N} \frac{x^m}m
	\end{equation}
	valid for all $|x|<1$
	and hence the latter lemma implies
	\begin{equation*}
	\log  D_n
	= -\sum\limits_{m\in\N} \frac{\Tr (H_n^m)}{m}.
	\end{equation*}
The above series is convergent because of Lemma~\ref{l:Largest_eig} 
	and consists of negative terms only. Thus any finite truncation of the series gives an upper bound on the series.

	For the upper bound we rely on the following result proved in \cite[Thm. 4.3 and Cor. 4.4]{Wi:66}.
	\begin{lemma}\label{l:moments}\cite[Theorem 4.3]{Wi:66}
		For any fixed $m\in\N$ the moments of $H_n$ satisfy
		\begin{equation}\label{e:moments}
		\lim_{n\to\infty} \frac{\Tr (H_n^m)} {\log n} = \frac{1}{\pi} 
		\int\limits_{0}^{\infty}  \sech^m\left(u \pi\right) \de{u} 
		=:\mu_m.
		\end{equation}
	\end{lemma}
	
	\begin{remark}
		Compared to the result stated in \cite{Wi:66}, we have an additional prefactor $1/2\pi$ in \eqref{e:moments}.
		This was missed in \cite{Wi:66} and pointed out to us by A.~Pushnitski and E.~Fedele, see \cite{F:19}.
	\end{remark}
	
	From the latter lemma we readily obtain an upper bound on the determinant $D_n$:
	
	\begin{corollary}[Upper bound]
		We obtain the bound 
		\begin{equation*}
		\limsup_{n\to\infty} \frac{\log D_n}{\log n}  \leq - \sum_{m\in\N} \frac{\mu_m}m= \frac 1 \pi \int_0^\infty \log\big(1- \sech(\pi u)\big) \de{u}. 
		\end{equation*}
	\end{corollary}	
	Next, we prove a lower bound.
	Since the matrix elements of $H$ are all positive we obtain the inequality
	\begin{equation*}
	\Tr\big( H_n^m\big) \leq \Tr\big( 1_n H^m 1_n\big),
	\end{equation*}
	where $1_n$ denotes the projection on $\ell^2(\{0,1,...n-1\})\subset \ell^2(\N_0)$. Hence we obtain for all $\eps>0$
	\begin{align}
	\log D_n 
	&\geq -\sum_{m\in\N} \frac{\Tr\big( 1_n H^m 1_n\big)}m\notag\\
	&\geq 
	-\sum_{m\in\N} \frac{\Tr\big( 1_n H^m 1_{>\varepsilon} (H) 1_n\big)}m - \sum_{m\in\N} \frac{\Tr\big( 1_n H^m 1_{\leq\varepsilon} (H)1_n\big)}m.			\label{e:lower_split}
	\end{align}
	Here $1_{>\eps}(H)$ is short for the spectral projection of $H$ on the set $\{x: x>\eps\}$. We first estimate the second term.
	
	\begin{lemma}\label{l:secondtry23}
		We obtain
		\begin{equation*}
		\limsup_{\varepsilon\to 0}\limsup_{n\to\infty}\frac1{\log n}\sum_{m\in\N} 
		\frac{\Tr\big( 1_n H^m 1_{\leq\varepsilon} (H)1_n\big)}m =0. 
		\end{equation*}
	\end{lemma}
	We prove this lemma in Subsection \ref{pf:lemma3}.
	In order to treat the first term on the left hand side of \eqref{e:lower_split} 
	we use the explicit diagonalization of the operator $H$ found in \cite[Sec. 4]{KS:16}. 
	\begin{lemma}
		We define for $l\in\N_0$ 
		\begin{align}\label{e:defPn}
		\hat P_{l} (x^2) 
		&:= \frac{4^l \big( \frac 1 4\big)_{(l)} \big( \frac 1 2\big)_{(l)} \big( \frac 3 4\big)_{(l)}}{l! \big(\frac 1 2\big)_{(2l)}} \textsubscript 4 F_3 \Big(-l, l+\frac 1 2, i 		\frac x 2, -i 	\frac x 2; \frac 	1 4, \frac 1 2, \frac 3 4;1\Big), 
		\end{align}
		where $\textsubscript p F_q$ denotes the generalized hypergeometric 
		function. Moreover, let
		\begin{equation*}
		\rho(x) := 2\sech(\pi x).
		\end{equation*}
		and $\mathcal H:= L^2((0,\infty), d \rho)$. Then $U:\ell^2(\N_0)\to \mathcal H$ 
		defined by $(U e_l)(x):= \hat P_{l}(x^2)$ is a unitary and
		\begin{equation*}
		(U H U^*f)(x) = \sech(\pi x)f(x),\qquad x>0,
		\end{equation*}	
		i.e. $U$ diagonalizes the operator $H$. 	
	\end{lemma}

	The latter lemma, Fubini's theorem and \eqref{e:Log_series} readily imply
		\begin{align}\label{e:trace_eq}
		&\sum_{m\in\N}\frac 1 m  \Tr\big(1_n H^m 1_{>\eps}(H) 1_n \big) \notag\\
		&= 
		2\sum_{m\in\N}\frac 1 m  \int_0^\infty \big(\sech(\pi x)\big)^m 1_{>\eps}(\sech(\pi x)) \sech(\pi x) \sum_{l=0}^{n-1} \big|  \hat P_l(x^2)\big|^2\de{x}\notag\\
		&=
		 -\int_0^{\frac{\sech^{-1}(\eps)}\pi}  \log\big( 1- \sech(\pi x)\big)  2 \sech(\pi x) \sum_{l=0}^{n-1} \big|  \hat P_l(x^2)\big|^2 \de{x}.
	\end{align}
	Next we are interested in the asymptotic behaviour of $ \hat P_l(x^2)$ as $l\to\infty$. 
	\begin{lemma}\label{l:Pl_asympt}
		Let $x>0$ be fixed. Then as $l\to\infty$ the asymptotics
		\begin{equation*}
		\hat P_l(x^2) =\sqrt{\frac{\cosh(\pi x)}{\pi l}} 
		\cos\left( x \log l + \arg A(i x/2)\right)
		\left(1+O(l^{-1})\right) + l^{-3/2}  R_l\left(x\right)
		\end{equation*}
		holds, where $\arg$ is the argument function, 
		\begin{equation*}
		A(i  x/2 ):=\frac  {2^{2 i x - 3/2} \cosh(\pi x)^{1/2} }{\pi^{3/2}}
		\end{equation*}
		and the $O$-term is independent of $x>0$ and
		\begin{equation*}
		\sup_{x\in[0,M]}\sup_{l\in \N_0} \left|R_l\left(x\right)\right| \leq r\left(M\right)
		\end{equation*}
		for some constant $r(M)$ depending on $M>0$. 
	\end{lemma}
We prove this in Subsection \ref{pf:lemma4}. 
	\begin{lemma}\label{l:secondtry2} 
		Let $\eps>0$. Then we obtain the upper bound
		\begin{align*}
		\limsup_{n\to\infty} \frac{\sum_{m\in\N} \frac 1 m  \Tr\big(1_n1_{>\eps}(H) H^m 1_n \big)}{\log n} 
		\leq 
		-\frac 1 \pi \int_0^\infty   \log (1- \sech\left(\pi x\right))   \de{x} .  
		\end{align*}
	\end{lemma}
We prove this in Subsection \ref{pf:lemma4}. 
The bound \eqref{e:lower_split} implies that
		\begin{align*}
		\liminf_{n\to\infty} \frac{\log D_n}{\log n} 
		\geq
		& -\limsup_{\eps\to 0} 
		\limsup_{n\to\infty}\frac 1 {\log n}\notag\\
		&\times\Big(\sum_{m\in\N} \frac{\Tr\big( 1_n H^m 1_{>\varepsilon} (H) 1_n\big)}m + \sum_{m\in\N} \frac{\Tr\big( 1_n H^m 1_{\leq\varepsilon} (H)1_n\big)}m\Big).
		\end{align*}
	From Lemma \ref{l:secondtry23}  and Lemma \ref{l:secondtry2} the assertion follows. 
	
\end{proof}

\section{Proof of Lemma~\ref{l:CJPD}}\label{pf:lemma1}

In this Section we assume that if the region of integration is not specified, it is over
whole space of corresponding dimension.
\begin{lemma}\label{l:MC_distr}
	Let $O \in O\left(N+\ell\right)$ be an orthogonal matrix drawn randomly
	according to \eqref{e:measure_O}. Then the joint distribution of the
	top left minor $M_N$ and bottom left minor $C_{\ell\times N}$ (cf. \eqref{e:O_split})
	is given by
	\begin{equation}\label{e:MC_distr}
		\de{P\left(M_N,C_{\ell\times N}\right)} = \frac{v_{\ell}}{v_{N+\ell}} 
		\delta \left(M_N^TM_N+ C_{\ell \times N}^T C_{\ell \times N} - I_N\right)
		\De{M_N}\De{C_{\ell \times N}}
	\end{equation}	
	where $\De{M_N}$ and $\De{C_{\ell \times N}}$ denote Lebesgue measure on $\R^{N\times N}$, respectively $\R^{\ell\times N}$.
\end{lemma}
\begin{proof}[Proof of Lemma~\ref{l:MC_distr}]
	Block decomposition of the matrix $O$ yields
	\begin{equation*}
		O^TO = 
		\left(
		\begin{array}{cc}
			M_N^TM_N+ C_{\ell \times N}^T C_{\ell \times N} & 
			M_N^TB_{N\times \ell}+ C_{\ell \times N}^T D_{\ell}\\
			B_{N \times \ell}^TM_N+ D^T_{\ell} C_{\ell \times N} & 
			B_{N \times \ell}^T B_{N \times \ell}+D_{\ell}^T D_{\ell}
		\end{array}
		\right),
	\end{equation*}
	and the corresponding measure on the set of tuples $M_N,C_{\ell\times N}$ now can be rewritten
	as
	\begin{multline*}
		\frac{\de{P\left(M_N,C_{\ell\times N}\right)}}{\De{M_N}\De{C_{\ell \times N}}} =
		\frac{1}{v_{N+\ell}} 
		\int \delta \left(	M_N^TM_N+ C_{\ell \times N}^T C_{\ell \times N} - I_N\right) 
		\\
		\delta \left(M_N^TB_{N\times \ell}+ C_{\ell \times N}^T D_{\ell}\right)
		\delta \left(B_{N \times \ell}^T B_{N \times \ell}+D_{\ell}^T D_{\ell}-I_{\ell}\right)
		\de{B_{N \times \ell}}\de{D_{\ell}}.
	\end{multline*}
	Integration over $dB_{N \times \ell}$ gives
	\begin{multline*}
		\frac{\de{P\left(M_N,C_{\ell\times N}\right)}}{\De{M_N}\De{C_{\ell \times N}}} = \frac{1}{v_{N+\ell}} 
		\int \delta \left(	M_N^TM_N+ C_{\ell \times N}^T C_{\ell \times N} - I_N\right) 
		\\			  
		\delta \left(D_{\ell}^T C_{\ell \times N}M^{-1}_N M^{-T}_NC^T_{\ell \times N} 
		D_{\ell}+D_{\ell}^T D_{\ell}-I_{\ell}\right)
		\det^{-\ell} M_N \de{D_{\ell}}.
	\end{multline*}	
	A change of variables 
	$D_{\ell} = \left(I_{\ell}+C_{\ell\times N}M^{-1}_NM^{-T}_NC^T_{\ell \times 
		N}\right)^{-1/2} \hat{D}_{\ell}$ yields
	\begin{multline*}
		\frac{\de{P\left(M_N,C_{\ell\times N}\right)}}{\De{M_N}\De{C_{\ell \times N}}}
		= \frac{1}{v_{N+\ell}} \int
		\delta \left(M_N^TM_N+ C_{\ell \times N}^T C_{\ell \times N} - I_N\right)
		\\
		\delta \left(\hat{D}_{\ell}^T\hat{D}_{\ell} - I_{\ell}\right)
		\det^{-\ell/2} \left(I_{\ell}+C_{\ell\times N}M^{-1}_NM^{-T}_NC_{\ell \times 
			N}^T\right)
		\det^{-\ell} M_N
		\de{\hat{D}_{\ell}}.
	\end{multline*}
	Integrating over $\hat{D}_{\ell}$ gives $v_{\ell}$, see
	\eqref{e:measure_OO}. The condition $M_N^TM_N+ 
	C_{\ell \times N}^T C_{\ell \times N} = I_N$ implies that the two determinants in the latter integrand cancel which can be seen as follows
	\begin{align*}
		\det \left(I_{\ell}+C_{\ell\times N}M^{-1}_NM^{-T}_NC_{\ell \times N}^T\right)
		&=
		\det \left(I_{\ell}+C_{\ell\times N}
		(I_{N} - C_{\ell \times N}^TC_{\ell \times N})^{-1}C_{\ell \times N}^T
		\right)
		\\
		&= \det \left(I_{\ell}-C_{\ell \times N}C_{\ell \times N}^T\right)^{-1}\\
		&= \det^{-1}\left(I_N-C_{\ell \times N}^TC_{\ell \times N}\right)
		= \det^{-2} M_N.
	\end{align*}	
	And the statement is proved.
\end{proof}

\begin{proof}[Proof of Lemma~\ref{l:CJPD}]
	We start with the distribution of $M_N$ given by \eqref{e:denMN} (cf. \eqref{e:MC_distr})
	and try to integrate out all variables of $M_N$ except of its eigenvalues.
	Below we follow a method described by Edelman \cite{Ed:97}.
	
	We first note that any matrix can be uniquely written as a special 
	product called real Schur decomposition. 
	Let 
	\begin{equation*}
		\mathrm{spec }\, M_N = \left\{\lambda_1,\lambda_2,\ldots\lambda_L, z_1,\overline{z_1},
		z_2,\overline{z_2},\ldots,z_M,\overline{z_M} \right\},
	\end{equation*}
	be the ordered set of eigenvalues of $M_N \in \mathcal{X}_{L,M}$. 
	Any such matrix can be uniquely decomposed into a product of real matrices 
	\begin{equation*}
	M_N = O
	\left(
	\begin{array}{cccccccc}
	\lambda_1 & R_{1,2} & .. & R_{1,L} & R_{1,L+1} & .. & .. & R_{1,L+M} \\
	0 & \lambda_2 & .. & R_{2,L} & R_{2,L+1} & .. & .. & R_{2,L+M} \\
	\vdots & .. & .. & \vdots & .. & .. & .. & \vdots \\
	0 & .. & 0 & \lambda_L & R_{L,L+1} & ..& .. & R_{L,L+M} \\
	0 & .. & .. & 0 & \Lambda_{1} & R_{L+1, L+2} &.. & R_{L+1,L+M} \\
	0 & .. & .. & 0 & 0 & \Lambda_2 & .. & 
	R_{L+2,L+M} \\
	\vdots & .. & .. & \vdots & \vdots & .. & .. & \vdots \\
	0 & .. & .. & 0 & 0 & .. & 0 & \Lambda_M \\				
	\end{array}
	\right)
	O^T,
	\end{equation*}
	for an orthogonal matrix $O \in O\left(N\right)$. Here
	$\Lambda_1,\Lambda_2,\ldots,
	\Lambda_M$ are $2\times 2$ blocks of the form
	\begin{equation*}
	\left(
	\begin{array}{cc}
	\alpha_j & \beta_j\\
	-\gamma_j & \alpha_j
	\end{array}
	\right)		
	\end{equation*}
	with $\beta_j \geq \gamma_j$ and $\beta_j\gamma_j > 0$,
	corresponding to complex eigenvalues $z_j,\overline{z_j} 
	= \alpha_j \pm i\sqrt{\beta_j\gamma_j}$ ordered in
	lexicographical order, and finally
	\begin{equation*}
	R_{j,k} = \left\{
	\begin{array}{ccl}
	1\times 1 & \mbox{blocks} & j<k\leq L,\\
	1\times 2 & \mbox{blocks} & j\leq L < k,\\
	2\times 2 & \mbox{blocks} & L<j<k.
	\end{array}
	\right.
	\end{equation*}
	\begin{remark}
		In fact the decomposition is not unique. One can change
		$O\mapsto O \cdot O_d$, where $O_d$ is a diagonal orthogonal 
		matrix. However, we will make it unique by assuming that all elements
		in the first row of $O$ are positive.
	\end{remark}
	Next we use a result of Edelman regarding Schur decomposition 
	\begin{theorem}\label{t:Jacobian}[Theorem~5.1, \cite{Ed:97}]
		Let $M_N$ be written in form of the Schur decomposition
		$M_N = O \left(Z+R\right)O^T$. The Jacobian for the change
		of variables is given by
		\begin{align*}
		\De{M_N} = 2^M
		&\prod\limits_{1\leq j<k \leq L}
		\left|\lambda_j-\lambda_k\right|
		\prod\limits_{\substack{1\leq j\leq L \\ 1\leq k \leq M}}
		\left|\lambda_j-z_k\right|^2
		\prod\limits_{1\leq j<k \leq M}
		\left|z_j-z_k\right|
		\left|z_j-\overline{z_k}\right| 
		\\
		&\prod\limits_{1\leq j\leq M}\left(\beta_j-\gamma_j\right)
		\prod\limits_{j=1}^L
		\de{\lambda_j}
		\prod\limits_{k=1}^M
		\de{\Lambda_k}
		\De{R}\De{_H O},
		\end{align*}
		where $z_j,\overline{z_j}$ is the pair of eigenvalues
		corresponding to the block $\Lambda_j$,
		\begin{equation*}
		\de{\Lambda_j} = \de{\alpha_j}\de{\beta_j}\de{\gamma_j},
		\end{equation*} 
		$\De{R}$ is the product over all $N^2/2-N/2-M$ real parameters
		of $R$ and $\De{_H O}$ is the product
		over all independent elements of the antisymmetric matrix
		$O^T\De{O}$, which is the natural element of integration 
		(for Haar measure) over the space of orthogonal matrices.
	\end{theorem}
	For any function $\widehat{F}: \mathcal{M}\left(\R^N\right) \to \C$, where $\mathcal M(\R^N)$ is the set of all $N\times N$ matrices, 	its average 
	over $\mathcal{O}^{\ell}_N$ is given by
	\begin{equation*}
	\big\langle \widehat{F} \big\rangle=		
	\frac{v_{\ell}}{v_{N+\ell}} 
	\int
	\widehat{F}\left(M_N\right)	
	\delta \left(M_N^TM_N+ C_{\ell \times N}^T C_{\ell \times N} - I_N\right)
	\De{M_N}\De{C_{\ell \times N}}.
	\end{equation*}
	This average can be naturally splitted into averages over the disjoint sets 
	$\mathcal{X}_{L,M}$.
	If one assumes that $\widehat{F}$ depends
	only on the eigenvalues, i.e. it can be written as
	\begin{equation*}
	\widehat{F}\left(M_N\right) = F\left(\lambda_1,\ldots,\lambda_L,
	x_1+iy_1,\ldots,x_M+iy_M\right)
	\end{equation*}
	for some function $F:\R^L\times \C_+^M \to \C$, then 
	the joint conditional eigenvalue distribution $p^{\left(L,M\right)}_{\ell}$ is 
	defined by the identity
	\begin{equation*}
		\big\langle \widehat{F} \big\rangle_{\mathcal{X}_{L,M}} = 
		\int\limits_{\R^L}\int\limits_{\C_+^M}
		F\left(\vec{\lambda},\vec{Z}\right) p^{\left(L,M\right)}_{\ell}
		\left(\vec{\lambda},\vec{Z}\right)\de{\vec{\lambda}}\de{\vec{Z}}.
	\end{equation*}
	We start with
	changing variables from $M_N$ to the triplet $\left(O,Z,R\right)$ and use
	Theorem~\ref{t:Jacobian} to get
	\begin{align}
	\big\langle \widehat{F} \big\rangle_{\mathcal{X}_{L,M}}
	=\ & 2^M\frac{v_{\ell}}{v_{N+\ell}} 
	\int\limits_{O\left(N\right)}\De{_H O}
	\int\limits_{\R^{\binom{N}{2}-M}} \De{R}
	\int\limits_{\R^{\ell N}}
	\De{C_{\ell \times N}}
	\int\limits_{\R^L_>} 
	\de{\vec{\lambda}}\int\limits_{\R^M_>\times\left(\R_>^2\right)^{M}}\de{\vec{\Lambda}}	
	\notag
	\\
	&	
	\prod\limits_{k=1}^M 
	\left(\beta_k-\gamma_k\right)
	F\left(Z\right)
	\Delta\left(Z\right)
	\delta \left(\left(Z^T+R^T\right)\left(Z+R\right) + 
	C_{\ell \times N}^TC_{\ell \times N} - I_N\right),
	\label{e:F_average}
	\end{align}
	where $\R_>^L:=\big\{(x_1,...,x_L)\subset \R^L: x_1\leq x_2\leq...\leq x_L\big\}$ and $(\R^2_{>})^M$ is defined accordingly and
	\begin{equation*}
	\Delta\left(Z\right) = \prod\limits_{1\leq j<k \leq L}
	\left|\lambda_j-\lambda_k\right|
	\prod\limits_{\substack{1\leq j\leq L \\ 1\leq k \leq M}}
	\left|\lambda_j-z_k\right|^2
	\prod\limits_{1\leq j<k \leq M}
	\left|z_j-z_k\right|
	\left|z_j-\overline{z_k}\right|.
	\end{equation*}
	Now integration over the orthogonal group gives the prefactor $v_N$. Next we integrate over $R$.
	\begin{proposition}\label{p:R_integration}
	Integration over $R$ in \eqref{e:F_average} gives
		\begin{align}
			\big\langle \widehat{F} \big\rangle_{\mathcal{X}_{L,M}}
			=& \ 2^M\frac{v_{\ell}v_N}{v_{N+\ell}} 
			\int\limits_{\R^L_>} \de{\vec{\lambda}}
			\int\limits_{\R^{\ell N}}
			\De{C_{\ell \times N}}	
			\int\limits_{\R^M_>\times\left(\R_>^2\right)^{M}}\de{\vec{\Lambda}}
			F\left(Z\right)
			\Delta\left(Z\right)						\notag
			\\		
			&
			\prod_{j=1}^L\left|\lambda_j\right|^{j-N}\delta\left(\lambda_j^2+C_j^TX_jC_j-1\right)
			\prod\limits_{k=1}^M
			\left(\beta_k-\gamma_k\right)\left(\det^{-2} \Lambda_k\right)^{M-k}				\notag
			\\		
			&\prod_{k=1}^M\delta\left(\Lambda_k^T\Lambda_k+
			\begin{pmatrix}
				C_{L+2k-1}^T \\
				C_{L+2k}^T
			\end{pmatrix}
			Y_{k}
			\begin{pmatrix}
				C_{L+2k-1} \quad C_{L+2k}
			\end{pmatrix}
			-I_2
			\right),		\label{Integration:eq1}
		\end{align}
		where $C_i$ is the $i$'s column of the matrix $C_{\ell\times N}$
		and $X_j$, $Y_j$ are $\ell\times\ell$ real symmetric, positive definite
		matrices defined recursively  by
		\begin{equation*}
		\begin{cases}
		X_1 &= I_{\ell},\\
		X_{k+1} &= X_k+X_kP_kX_k, \quad k=\overline{1,L-1},
		\end{cases}
		\end{equation*}
	respectively
		\begin{equation*}
		\begin{cases}
		Y_1 &= I_{\ell}+\sum_{j=1}^{L}X_jP_jX_j,\\
		Y_{k+1} &= Y_k+Y_kQ_kY_k, \quad k=\overline{1,M-1}
		\end{cases}
		\end{equation*}	
		with $P_{j} = \frac{C_jC_j^T}{\lambda_j^2}$ and 
		$Q_j = 
		\begin{pmatrix}
		C_{L+2j-1} \quad C_{L+2j}
		\end{pmatrix}
		\Lambda_j^{-1}
		\Lambda_j^{-T}
		\begin{pmatrix}
		C_{L+2j-1} \quad C_{L+2j}
		\end{pmatrix}^T.
		$
		Moreover,
		\begin{equation*}
			\det X_1 = 1, \quad \det X_{k+1} = \lambda_k^{-2} \det X_k.
		\end{equation*}
		\begin{equation*}
			\det Y_{k+1} = \det^{-2} \Lambda_k\det Y_k.
		\end{equation*}	
	\end{proposition}
	\begin{proof}[Proof of Proposition~\ref{p:R_integration}]
		We start by rewriting the $\delta$-function in \eqref{e:F_average}
		element wise. The corresponding term can be written as a product of 
		lower dimensional, e.g. $1\times 1$,$1\times 2$, $2\times 2$, $\delta$-functions.
		For the top left corner we have $L$ $\delta$-functions of the form
		\begin{equation}\label{e:d_ii}
		\delta\left(\sum\limits_{j<k}R_{j,k}^2+\lambda_k^2+C_k^TC_k-1\right), \quad 1\leq k\leq L,
		\end{equation}
		and $\binom{L-1}{2}$ $\delta$-functions of the form
		\begin{equation}\label{e:d_ik}
		\delta\left(\sum\limits_{j=1}^{k-1}R_{j,k}R_{j,m} +\lambda_kR_{k,m}+C_k^TC_m\right), \quad 
		1\leq k < m\leq L.
		\end{equation}
		By solving equations corresponding to restrictions \eqref{e:d_ik} inductively one can see
		that $R_{k,m}$ can be expressed as $-\lambda_k^{-1}C_k^TX_{k}C_m$ for some matrix $X_{k}$. 
		Applying the ansatz one gets
		\begin{equation*}
			\sum\limits_{j=1}^{k-1} X_j^TP_jX_j - X_k + I_{\ell} = 0,
		\end{equation*}
		and the former is solved by the $X_j$'s defined above and therefore
		\begin{equation*}
			R_{k,m} = -\lambda^{-1}_k C_k^TX_kC_m.
		\end{equation*}
		The top right corner $\delta$-functions, corresponding to $1\times 2$ matrices, 
		can be written as
		\begin{equation*}
			\delta\left(			
			\sum\limits_{j=1}^{k-1}R_{j,k}R_{j,L+m}
			+\lambda_k R_{k,L+m}	
			+\left[C_k^TC_{L+2m-1} \, C_k^TC_{L+2m} \right]
			\right), \, 1\leq m \leq M.
		\end{equation*}
		These conditions can be resolved inductively in a similar way and one gets
		\begin{equation*}			
				R_{k,L+m} = -\frac{1}{\lambda_k}C_{k}^T X_k \left[C_{L+2m-1}
				\quad C_{L+2m}\right], \quad 1\leq k\leq L, 1\leq m \leq M.
		\end{equation*}
		Finally, the bottom right corner $\delta$-functions, corresponding to $2\times 2$ blocks,
		are given by
		\begin{multline*}
		\delta\left(\sum\limits_{j=1}^{L} R_{j,L+k}^T R_{j,L+m}
		+ \sum\limits_{j=1}^{k-1} R_{L+j,L+k}^T R_{L+j,L+m}
		+\Lambda_k^TR_{L+k,L+m} 
		\right.
		\\
		+\left.
		\begin{pmatrix}
		C_{L+2k-1}^TC_{L+2m-1} & C_{L+2k-1}^TC_{L+2m} \\
		C_{L+2k}^TC_{L+2m-1} & C_{L+2k}^TC_{L+2m}
		\end{pmatrix}
		\right), \quad 1\leq k<m \leq M.
		\end{multline*}
		with diagonal ones
		\begin{multline*}
		\delta\left(\sum\limits_{j=1}^{L} R_{j,L+m}^T R_{j,L+m}
		+\sum\limits_{j=1}^{m-1} R_{L+j,L+m}^T R_{L+j,L+m}
		+\Lambda_m^T\Lambda_m
		\right.
		\\
		\left.
		+
		\begin{pmatrix}
		C_{L+2m-1}^TC_{L+2m-1} & C_{L+2m-1}^TC_{L+2m} \\
		C_{L+2m}^TC_{L+2m-1} & C_{L+2m}^TC_{L+2m}
		\end{pmatrix}
		-I_2
		\right),
		\quad 1\leq m \leq M.
		\end{multline*}		
		The "non-diagonal equations" are solved by
		\begin{equation*}
			R_{L+k,L+m} = -\Lambda_k^{-T} 
			\begin{pmatrix}
			C_{L+2k-1}^T \\
			C_{L+2k}^T
			\end{pmatrix}
			Y_{k}
			\begin{pmatrix}
			C_{L+2m-1} \quad C_{L+2m}
			\end{pmatrix},
			\, 1\leq k<m \leq M.
		\end{equation*}
		where the $Y_k$ are defined above. Positiveness of the matrices $X_k,Y_k$ is
		obvious from their definitions. Finally, we integrate out
		the variables $R$ from \eqref{e:F_average} to obtain the first statement.
		We just mention that all $\delta$-functions constraints contained
		$R_{k,m}$ elements with a prefactor of the form either $\lambda_k$ or
		$\Lambda_k$ and this is the source of  Jacobian 
		\begin{equation*}
		\prod_{j=1}^L\left|\lambda_j\right|^{j-N}
		\prod\limits_{k=1}^M	\left(\det^{-2} \Lambda_k\right)^{M-k},
		\end{equation*}		
		appearing in the final result. For the determinants we note
		\begin{equation*}
		\det X_{k+1} = \det\left(X_k+X_kP_kX_k\right) = \det X_k\det\left(I+P_kX_k\right).
		\end{equation*}
		$P_k$ is up to a constant the projection onto $C_k$, and therefore the matrix $I+P_kX_k$
		can only have one eigenvalue different from $1$ with corresponding eigenvector $C_k$. Therefore,
		\begin{equation*}
		\det\left(I+P_kX_k\right) = 1+\lambda_k^{-2}\langle C_k, X_kC_k\rangle.
		\end{equation*}
		Finally, we note that there is a $\delta$-function of the form \eqref{e:d_ii}
		left in the integration with argument $\lambda_k^2-1+
		\langle C_k, X_kC_k\rangle$. Thus
		\begin{equation*}
		\det\left(I+P_kX_k\right) = 1+\frac{1-\lambda_k^2}{\lambda_k^2} = \lambda_k^{-2}.
		\end{equation*}
		For $Y_k$ the proof is analogous.
	\end{proof}	

	In the next step we integrate out the $C$ variables in \eqref{Integration:eq1}. We first change coordinates
	\begin{align*}
		&C_j\to X_j^{-1/2}C_j, \quad 1\leq j\leq L, \\
		&\left[C_{L+2j-1}\, C_{L+2j}\right]
		\to Y_j^{-1/2}\left[C_{L+2j-1}\, C_{L+2j}\right],\quad 1\leq j\leq M.
	\end{align*}
	It follows from Proposition~\ref{p:R_integration}, that all determinants 
	coming from Jacobians cancel and we obtain
	\begin{align*}
		\big\langle \widehat{F} \big\rangle_{\mathcal{X}_{L,M}}
		=\ & 2^M\frac{v_{\ell}v_N}{v_{N+\ell}} 
		\int\limits_{\R^L_>} 
		\de{\vec{\lambda}}
		\int\limits_{\R^M_>\times\left(\R_>^2\right)^{M}}
		\de{\vec{\Lambda}}\int\limits_{\R^{\ell N}}
		\De{C_{\ell \times N}}
		F\left(Z\right)
		\Delta\left(Z\right)
		\\
		&
		 \prod\limits_{k=1}^M
		\left(\beta_k-\gamma_k\right)\prod_{j=1}^L\delta\left(\lambda_j^2+C_j^TC_j-1\right)	
		\\	
		&
		\prod_{m=1}^M\delta\left(\Lambda_m^T\Lambda_m+
		\begin{pmatrix}
			C_{L+2m-1}^T \\
			C_{L+2m}^T
		\end{pmatrix}
		\begin{pmatrix}
			C_{L+2m-1} \quad C_{L+2m}
		\end{pmatrix}
		-I_2
		\right),
	\end{align*}
	Now we are ready to integrate out the variables $C$. 
	The first $L$ integrals are given by Proposition~\ref{p:delta_m}
	and using the corresponding result we integrate over the 
	variables $C_i$ for $i=\overline{1,L}$ 
	\begin{align*}
		\big\langle \widehat{F} \big\rangle_{\mathcal{X}_{L,M}}
		=\ & 2^M\frac{v_{\ell}v_N \pi^{L\ell/2}}{v_{N+\ell}\Gamma^L \left(\frac{\ell}{2}\right)} 
		\int\limits_{\R^L_>} \de{\vec{\lambda}}
		\int\limits_{\R^M_>\times\left(\R_>^2\right)^{M}}
		\de{\vec{\Lambda}}
		\int\limits_{\R^{\ell \left(N-L\right)}}
		\De{C_{\ell \times \left(N-L\right)}}	
		F\left(Z\right) \Delta\left(Z\right)		
		\\
		&
		\int\limits_{\R^{\ell \left(N-L\right)}}
		\De{C_{\ell \times \left(N-L\right)}}	
		\prod\limits_{j=1}^L\left(1-\lambda_j^2\right)^{\frac{\ell}{2}-1}_+		
		\prod_{k=1}^M \left(\beta_k-\gamma_k\right)	
		\\
		&			
		\prod_{m=1}^M\delta\left(\Lambda_m^T\Lambda_m+
		\begin{pmatrix}
			C_{L+2m-1}^T \\
			C_{L+2m}^T
		\end{pmatrix}
		\begin{pmatrix}
			C_{L+2m-1} \quad C_{L+2m}
		\end{pmatrix}
		-I_2
		\right),
	\end{align*}
	For every $2\times 2$ dimensional $\delta$-function we have $2\ell$ 
	variables of integration and $3$ $\delta$-functions. Therefore we need to distinguish 
	two different cases: $\ell=1$ (singular) and $\ell\geq 2$. In the first case
	we have
	\begin{equation*}
	I\left(\Lambda\right):=\int\de{x}\de{y}\delta\left(x^2-\Lambda_{11}\right)
	\delta\left(y^2-\Lambda_{22}\right)\delta\left(xy-\Lambda_{12}\right)
	\delta\left[\det\left(I_2-\Lambda^T\Lambda\right)\right],
	\end{equation*}
	where
	$
	\begin{pmatrix}
	\Lambda_{11} & \Lambda_{12} \\
	\Lambda_{12} & \Lambda_{22}
	\end{pmatrix} = I_2-\Lambda^T\Lambda.
	$ For $\ell\geq 2$ we have a non-singular integral of the form
	\begin{equation*}
	I\left(\Lambda\right):=\int\de{\vec{x}}\de{\vec{y}}
	\delta\left(\left\|\vec{x}\right\|^2-A\right)
	\delta\left(\left\|\vec{y}\right\|^2-C\right)
	\delta\left(\left\langle\vec{x},\vec{y}\right\rangle-B\right).
	\end{equation*}
	We change variables $\vec{y}$ to 
	$t = \left\|P_{\vec{x}}\vec{y}\right\| 
	= \frac{\left\langle \vec{x},\vec{y} \right\rangle}{\left\|\vec{x}\right\|^2}$
	and $\vec{r} = \left(r_1,r_2,\ldots,r_{\ell-1}\right)$ such that
	\begin{equation*}
		\vec{y} = t\vec{x}+\Big(r_1,\ldots,r_{\ell-1},-\frac{1}{x_{\ell}}\sum\limits_{j=1}^{\ell-1}x_jr_j\Big).
	\end{equation*}
	Then the corresponding Jacobian is given by $J = \frac{\left\|\vec{x}\right\|^2}{x_{\ell}}$. Applying this to 
	the above, yields
	\begin{align*}
		I\left(\Lambda\right)=&\int\de{\vec{x}}\de{\vec{r}}\de{t}
		\delta\big(\left\|\vec{x}\right\|^2-\Lambda_{11}\big)
		\delta\big(t\left\|\vec{x}\right\|^2-\Lambda_{12}\big) \frac{\left\|\vec{x}\right\|^2}{\left|x_{\ell}\right|}
		\delta\Big(t^2\left\|\vec{x}\right\|^2+\left\|\vec{r}\right\|^2
		+x_{\ell}^{-2}\Big(\sum\limits_{j=1}^{\ell-1}x_jr_j\Big)^2-\Lambda_{22}\Big)
		\\
		=& \int\de{\vec{x}}\de{\vec{r}}
		\delta\big(\left\|\vec{x}\right\|^2-\Lambda_{11}\big)
		\delta\left( \left\langle \vec{r}, V\vec{r} \right\rangle - U\right)
		\left|x_{\ell}\right|^{-1}.
	\end{align*}
	where $V = I_{\ell-1}+x_{\ell}^{-2}\left(x_1,x_2,\ldots,x_{\ell-1}\right)^T\left(x_1,x_2,\ldots,x_{\ell-1}\right)$
	is a positive definite matrix of size $\left(\ell-1\right)\times \left(\ell-1\right)$
	and $U = \Lambda_{22}-\frac{\Lambda_{12}^2}{\left\|\vec{x}\right\|^2}$. If $U<0$, then the integral with respect to $\vec{r}$ vanishes.
	Otherwise we get after changing variables $\vec{r} = V^{-1/2} \vec{s}$ and applying the result of 
	Proposition~\ref{p:delta_m}
	\begin{equation*}
	I\left(\Lambda\right) = 
	\frac{\pi^{\frac{\ell-1}{2}}}{\Gamma\left(\frac{\ell-1}{2}\right)}	
	\int\de{\vec{x}}
	\delta\left(\left\|\vec{x}\right\|^2-\Lambda_{11}\right)
	U_+^{\frac{\ell-3}{2}}
	\det^{-1/2} V \left|x_{\ell}\right|^{-1}.
	\end{equation*}
	It is easy to see that 
	\begin{equation*}
	\det V = 1+\Big(\sum\limits_{j=1}^{\ell-1}x_j^2\Big)\big/x_{\ell}^2
	= \frac{\left\|\vec{x}\right\|^2}{x_{\ell}^2}.
	\end{equation*}
	Then
	\begin{align*}
	I\left(\Lambda\right) &= 
	\frac{\pi^{\frac{\ell-1}{2}}}{\Gamma\left(\frac{\ell-1}{2}\right)}	
	\int\de{\vec{x}}
	\delta\left(\left\|\vec{x}\right\|^2-\Lambda_{11}\right)
	\left(\Lambda_{22} \left\|\vec{x}\right\|^2 - \Lambda_{12}^2\right)_+^{\frac{\ell-3}{2}}
	\left\|\vec{x}\right\|^{2-\ell}
	\\	
	&
	= \frac{2^{\ell-2}\pi^{\ell-1}}{\Gamma\left(\ell-1\right)}
	\det\left(I_2-\Lambda^T\Lambda\right)_+^{\frac{\ell-3}{2}},
	\end{align*}
	where we used one more time Proposition~\ref{p:delta_m}. Summarizing the above, we have shown for	 $\ell=1$
		\begin{align*}
			\big\langle \widehat{F} \big\rangle_{\mathcal{X}_{L,M}}
			=& 2^{M}\frac{v_1v_N }{v_{N+1}} 
			\int\limits_{\R^L_>} \de{\vec{\lambda}}			
			F\left(Z\right) \Delta\left(Z\right)
			\prod\limits_{j=1}^L\left(1-\lambda_j^2\right)^{-\frac{1}{2}}_+
			\\		
			&\int\limits_{\R^M_>\times\left(\R_>^2\right)^{M}}
			\de{\vec{\Lambda}}
			\prod\limits_{k=1}^M
			\left(\beta_k-\gamma_k\right)
			\delta\left(\det\left(I_2-\Lambda_k^T\Lambda_k\right)\right),
		\end{align*}
		while for $\ell\geq 2$
		\begin{align*}
			\big\langle \widehat{F} \big\rangle_{\mathcal{X}_{L,M}}
			=& \frac{2^{M\left(\ell-1\right)} \pi^{N\ell/2-M}  v_{\ell}v_N }{
				v_{N+\ell}\Gamma^L \left(\frac{\ell}{2}\right)\Gamma^M \left(\ell-1\right)} 
			\int\limits_{\R^L_>} \de{\vec{\lambda}}
			F\left(Z\right) \Delta\left(Z\right)
			\prod\limits_{j=1}^L\left(1-\lambda_j^2\right)^{\frac{\ell}{2}-1}_+
			\\
			&\int\limits_{\R^M_>\times\left(\R_>^2\right)^{M}}
			\de{\vec{\Lambda}}
			\prod\limits_{k=1}^M
			\left(\beta_k-\gamma_k\right)
			\det^{\frac{\ell-3}{2}}\left(I_2-\Lambda_k^T\Lambda_k\right)_+.
		\end{align*}
	Any real symmetric $2\times 2$ matrix is non-negative iff
	$	
	\Tr M \geq 0,$ and$
	\det M\geq 0.
	$
	Therefore, condition $\Lambda_k^T\Lambda_k\leq I_2$ is equivalent to
	\begin{equation*}
	\begin{cases}
	2\left(1-\alpha_k^2-\beta_k\gamma_k\right)- \left(\beta_k-\gamma_k\right)^2\geq 0 \\
	\left(1-\alpha_k^2-\beta_k\gamma_k\right)^2 - \left(\beta_k-\gamma_k\right)^2\geq 0.
	\end{cases}
	\end{equation*}
	It is easy to see that the second condition implies the first one and therefore we suppress
	the first one. 
	The last step is a change of variables from $\left(\alpha_k,\beta_k,\gamma_k\right)$
	to $z_k = x_k+iy_k$. Let
	\begin{equation*}
	\begin{cases}
	x_k &= \alpha_k, \\
	y_k &= \sqrt{\beta_k\gamma_k}, \\
	\delta_{k} &=\beta_k-\gamma_k.
	\end{cases}	
	\end{equation*}
	This change of variables is two to one and its 
	Jacobian is equal to
	\begin{equation*}
	J_k=\frac{2 y_k}{\sqrt{4y_k^2+\delta_k^2}}.
	\end{equation*}
	Finally, for $\ell=1$ integrating out the variables $\delta_k$ we obtain
	\begin{equation*}
		\big\langle \widehat{F} \big\rangle_{\mathcal{X}_{L,M}}
		= \frac{2^{M}v_N }{v_{N+1}} 
		\int
		\de{\vec{\lambda}}
		\de{\vec{x}}
		\de{\vec{y}}
		F\left(Z\right) \Delta\left(Z\right)
		\prod\limits_{j=1}^L\left(1-\lambda_j^2\right)^{-\frac{1}{2}}_+		
		\prod\limits_{k=1}^M \frac{2y_k}{\left|1-z_k^2\right|}.
	\end{equation*}	
	For $\ell\geq 2$ let $\Omega_k = \left\{\left(\alpha_k,\beta_k,\gamma_k\right):
	\beta_k\geq\gamma_k \wedge  \left(\beta_k-\gamma_k\right)^2\leq
	\left(1-\alpha_k^2-\beta_k\gamma_k\right)^2\right\}$, then
	\begin{align*}
	&\int\limits_{\Omega_k}
	\left(\beta_k-\gamma_k\right)
	\left(\left(1-\alpha_k^2-\beta_k\gamma_k\right)^2 - 
	\left(\beta_k-\gamma_k\right)^2\right)^{\frac{\ell-3}{2}}
	\de{\alpha_k}\de{\beta_k}\de{\gamma_k}
	\\
	&=
	\int\limits_{\R^M_>}\de{\vec{x}}
	\int\limits_{\R^M_+}\de{\vec{y}}
	\prod\limits_{k=1}^M\int\limits_{0}^{\left|1-x_k^2-y_k^2\right|}\de{\delta_k}
	\frac{4 y_k\delta_k}{\sqrt{4y_k^2+\delta_k^2}}
	\left(\left(1-x_k^2-y_k^2\right)^2 - 
	\delta_k^2\right)^{\frac{\ell-3}{2}}
	\\
	&=
	4y_k\int\limits_{-\infty}^{\infty}\de{x_k}
	\int\limits_{0}^{\infty}\de{y_k}
	\left|1-z_k^2\right|^{\ell-2}
	\int\limits_{\frac{2y_k}{\left|1-z_k^2\right|}}^1
	\left(1-u^2\right)^{\frac{\ell-3}{2}}\de{u},
	\end{align*}
	where we made the substitution $u = \sqrt{4y^2_k+\delta_k^2}\left|1-z_k^2\right|^{-1}$ 
	in the last integral. Introducing
	\begin{equation*}
	w^2_{\ell}\left(z\right) = 
	\begin{cases}
	\displaystyle\frac{1}{2\pi}\left|1-z^2\right|^{-1},&\ell=1,\\
	\displaystyle\frac{\ell\left(\ell-1\right)}{2\pi}
	\left|1-z^2\right|^{\ell-2}\int\limits_{
		\frac{2\left|\Im z\right|}{\left|1-z^2\right|}}^{1} \left(1-u^2\right)^{\frac{\ell-3}{2}}\de{u},
	&\ell\geq 2,
	\end{cases}
	\end{equation*}
	one can rewrite the answer as stated in Lemma~\ref{l:CJPD}. We like to stress 
	that for real $z$ the weight function can be written for any $\ell$ as
	\begin{equation*}
		w^2_{\ell}\left(x\right) = \frac{\Gamma\left(\ell+1\right)}{2^{\ell}\Gamma^2\left(\frac{\ell}{2}\right)}
		\left|1-x^2\right|^{\ell-2},
	\end{equation*}
	which follows from calculating the integral
	\begin{equation*}
		\int\limits_{0}^1 \left(1-u^2\right)^{\frac{\ell-3}{2}}\de{u} 
		= \frac{1}{2}B\left(\frac{\ell-1}{2},\frac{1}{2}\right) 
		= \frac{\pi2^{1-\ell}\Gamma\left(\ell-1\right)}
		{\Gamma^2\left(\frac{\ell}{2}\right)}.
	\end{equation*}
\end{proof}

\section{Proof of Lemma~\ref{l:inner_pr}}\label{pf:lemma2}\begin{proof}[Proof of Lemma~\ref{l:inner_pr}]
	First we note that the skew-product of polynomials with indexes 
	of the same parity is zero: The real part of the skew-product 
	changes sign after changing $x\to-x$ and $y\to-y$ in the integral and hence has to vanish.
	For the complex part we first note that $w_{\ell}\left(x+iy\right)w_{\ell}\left(x-iy\right)$ is an even 
	function of $y$, therefore only odd powers (because of the additional $\sgn (y)$ factor)
	contribute to the integral when expanding $\pi_{j}\left(x+iy\right)\pi_k\left(x-iy\right)$
	in powers of $x$ and $y$. However, these term will contain odd powers of $x$ at the same time
	and thus will vanish when integrated over $\left[-1,1\right]$. For indexes with different 
	parities we begin by calculating $\left(\varepsilon \left[w\pi_{i}\right]\right)\left(x\right)$
	for even and odd values of $i$ and real $x$. For complex argument there is no need to calculate anything,
	as the transformed polynomials are different from the original ones by a factor $i\Im\left(z\right)$ 
	and complex conjugation only. Let $\left|x\right|\leq 1$ and $\ell\geq 2$, then 
		\begin{align*}
		\left(\varepsilon \left[w_{\ell} z^{2k}\right]\right)\left(x\right)
		&= \frac{1}{2} \int\limits_{\R}t^{2k}w_{\ell}\left(t\right)\sgn\left(t-x\right)\de{t}
		= -\sgn\left(x\right)
		\left(\frac{\ell!}{2^{\ell}\Gamma^2\left(\frac{\ell}{2}\right)}\right)^{1/2}
		\int\limits_{0}^{\left|x\right|}
		t^{2k}\left(1-t^2\right)^{\frac{\ell}{2}-1}\de{t}\\
		&
		=-\frac{1}{2}\left(\frac{\ell!}{2^{\ell}\Gamma^2\left(\frac{\ell}{2}\right)}\right)^{1/2}
		\sgn\left(x\right)B\left(x^2;k+\frac{1}{2},\frac{\ell}{2}\right)
	\end{align*}
	and
	\begin{align*}	
		\left(\varepsilon \left[w_{\ell} z^{2k+1}\right]\right)\left(x\right)
		&= \frac{1}{2} \int\limits_{\R}t^{2k+1}w_{\ell}\left(t\right)\sgn\left(t-x\right)\de{t}
		=
		\left(\frac{\ell!}{2^{\ell}\Gamma^2\left(\frac{\ell}{2}\right)}\right)^{1/2}
		\int\limits_{\left|x\right|}^{1}t^{2k+1}\left(1-t^2\right)^{\frac{\ell}{2}-1}\de{t}
		\\
		&=\frac{1}{2}\left(\frac{\ell!}{2^{\ell}\Gamma^2\left(\frac{\ell}{2}\right)}\right)^{1/2}
		\left[B\left(k+1,\frac{\ell}{2}\right) - B\left(x^2;k+1,\frac{\ell}{2}\right)\right].
	\end{align*}
	Using definitions of polynomials $\pi_j\left(z\right)$ together with the identity
	\begin{equation*}
		B\left(t;p+1,q\right) - \frac{p}{p+q} B\left(t;p,q\right) = -\frac{t^p\left(1-t\right)^q}{p+q},
	\end{equation*}
		\begin{align}
		\left(\varepsilon\left[w_{\ell}\pi_{2k}\right]\right)\left(z\right) &=
		\begin{cases}
			\displaystyle-\frac{1}{2}\left(\frac{\ell!}{2^{\ell}\Gamma^2\left(\frac{\ell}{2}\right)}\right)^{1/2}
			\sgn\left(x\right)B\left(x^2;k+\frac{1}{2},\frac{\ell}{2}\right), & x\in \R, \\
			i\sgn\left(\Im z\right)\overline{z}^{2k}, & z\in\C\setminus\R.
		\end{cases}
		\label{e:eps_pi_even}
		\\
		\left(\varepsilon\left[w_{\ell}\pi_{2k+1}\right]\right)\left(z\right) &= 
		\begin{cases}
			\displaystyle\frac{1}{2k+\ell}\left(\frac{\ell!}{2^{\ell}\Gamma^2\left(\frac{\ell}{2}\right)}\right)^{1/2}
			x^{2k}\left(1-x^2\right)^{\frac{\ell}{2}}, & x\in \R, \\
			i\sgn\left(\Im z\right)\overline{z}^{2k+1}, & z\in\C\setminus\R.
		\end{cases}
		\label{e:eps_pi_odd}
	\end{align}
	Now we proceed with the calculation of the skew-product.
	\begin{equation*}
	\left( \pi_{2j},\pi_{2k+1}\right)^{\ell} = \left( \pi_{2j},\pi_{2k+1}\right)_{\R}^{\ell}
	+\left( \pi_{2j},\pi_{2k+1}\right)_{\C}^{\ell}.
	\end{equation*}
	For the real part we use \eqref{e:eps_pi_even} and \eqref{e:eps_pi_odd} to write
	\begin{align*}
	\left( \pi_{2j},\pi_{2k+1}\right)_{\R}^{\ell} =&
	\frac{\ell!}{2^{\ell}\Gamma^2\left(\frac{\ell}{2}\right)}
	\int\limits_{-1}^{1} x^{2j}\left(1-x^2\right)^{\frac{\ell}{2}-1}
	\frac{x^{2k}}{2k+\ell}\left(1-x^2\right)^{\frac{\ell}{2}}\de{x}
	\\
	+&
	\frac{\ell!}{2^{\ell}\Gamma^2\left(\frac{\ell}{2}\right)}
	\int\limits_{-1}^{1}\frac{1}{2}
	\left(x^{2k+1}-\frac{2k}{2k+\ell}x^{2k-1}\right)
	\left(1-x^2\right)^{\frac{\ell}{2}-1}
	\sgn\left(x\right)B\left(x^2;j+\frac{1}{2},\frac{\ell}{2}\right)\de{x}.
	\end{align*}
	The first integral is given by $\frac{B\left(j+k+\frac{1}{2},\ell\right)}{2k+\ell}$, while for the second
	one we use the important observation
	\begin{equation*}
	\frac{\de{}}{\de{x}} \left(x^{2k}\left(1-x^2\right)^{\frac{\ell}{2}}\right)
	= -\left(2k+\ell\right)
	\left(x^{2k+1}-\frac{2k}{2k+\ell}x^{2k-1}\right)
	\left(1-x^2\right)^{\frac{\ell}{2}-1}.
	\end{equation*}
	Applying integration by parts to the second integral, we obtain
	\begin{align*}
	\left( \pi_{2j},\pi_{2k+1}\right)_{\R}^{\ell} 
	&=
	\frac{\ell!}{2^{\ell}\Gamma^2\left(\frac{\ell}{2}\right)\left(2k+\ell\right)}	
	\bigg(B\left(j+k+\frac{1}{2},\ell\right)
	+
	2\int\limits_{0}^{1}x^{2k+2j}\left(1-x^2\right)^{\ell-1}\de{x}	
	\bigg) 
	\\
	&= \frac{\ell!}{2^{\ell-1}\Gamma^2\left(\frac{\ell}{2}\right)\left(2k+\ell\right)}	
	B\left(j+k+\frac{1}{2},\ell\right).
	\end{align*}
	For the complex part of the skew-product and $\ell\geq 2$ symmetry of the integrand yields
	\begin{equation*}
		\left( \pi_{2j},\pi_{2k+1}\right)_{\C}^{\ell} = 2i
		\frac{\ell\left(\ell-1\right)}{2\pi}
		\int\limits_{\mathbb{D}}z^{2j}
		\left(\overline{z}^{2k+1}-\frac{2k}{2k+\ell}\overline{z}^{2k-1}\right)
		\sgn\left(\Im z\right)\left|w_{\ell}^2\left(z\right)\right| \de{^2 z}.
	\end{equation*}
	One can check that
	\begin{align*}
		\frac{\partial }{\partial x} \left(x-iy\right)^{2k}
		\left(1-\left(x-iy\right)^2\right)^{\frac{\ell}{2}} &=
		i\frac{\partial }{\partial y} \left(x-iy\right)^{2k}
		\left(1-\left(x-iy\right)^2\right)^{\frac{\ell}{2}}
		\\
		&= -\left(2k+\ell\right)\pi_{2k+1}\left(x-iy\right)\left(1-\left(x-iy\right)^2\right)^{\frac{\ell}{2}-1}.
	\end{align*}
	Using this, integration by parts with respect to $x$ implies
	\begin{align}
		\left( \pi_{2j},\pi_{2k+1}\right)_{\C}^{\ell} =& 
		\frac{2i}{2\pi\left(2k+\ell\right)}
		\int\limits_{\mathbb{D}} \overline{z}^{2k}
		\left(1-\overline{z}^2\right)^{\frac{\ell}{2}} 
		\frac{\partial}{\partial x} \left(z^{2j}
		\left(1-z^2\right)^{\frac{\ell}{2}-1}\right)
		\sgn\left(\Im z\right) \widehat{w}_{\ell}\left(z\right) \de{^2 z}\notag
		\\
		& + \frac{2i}{2\pi\left(2k+\ell\right)}
		\int\limits_{\mathbb{D}}
		\overline{z}^{2k}
		\left(1-\overline{z}^2\right)^{\frac{\ell}{2}} 
		z^{2j}	\left(1-z^2\right)^{\frac{\ell}{2}-1}
		\sgn\left(\Im z\right)
		\frac{\partial}{\partial x} \widehat{w}_{\ell}\left(z\right) \de{^2 z},\label{e:skew_compl_x}
	\end{align}
	where 
	$$\widehat{w}_{\ell}\left(z\right) = \ell\left(\ell-1\right)
	\int\limits_{\frac{2\left|\Im z\right|}{\left|1-z^2\right|}}^1
	\left(1-u^2\right)^{\frac{\ell-3}{2}}\de{u}.$$
	Integration by parts with respect to $y$ gives
	\begin{align}
		\left( \pi_{2j},\pi_{2k+1}\right)_{\C}^{\ell} =& -
		\frac{2}{2\pi\left(2k+\ell\right)}
		\int\limits_{\mathbb{D}} \overline{z}^{2k}
		\left(1-\overline{z}^2\right)^{\frac{\ell}{2}} 
		\frac{\partial}{\partial y} \left(z^{2j}
		\left(1-z^2\right)^{\frac{\ell}{2}-1}\right)
		\sgn\left(\Im z\right)\widehat{w}_{\ell}\left(z\right) \de{^2 z}\notag
		\\
		& - \frac{2}{2\pi\left(2k+\ell\right)}
		\int\limits_{\mathbb{D}}
		\overline{z}^{2k}
		\left(1-\overline{z}^2\right)^{\frac{\ell}{2}} 
		z^{2j}	\left(1-z^2\right)^{\frac{\ell}{2}-1}		
		\frac{\partial}{\partial y} \left(
		\sgn\left(\Im z\right)
		\widehat{w}_{\ell}\left(z\right)\right) \de{^2 z}\notag
		\\
		&-4\frac{\ell\left(\ell-1\right)}{2\pi\left(2k+\ell\right)}
		\int\limits_{-1}^{1}x^{2j+2k}\left(1-x^2\right)^{\ell-1}
		\int\limits_{0}^{1}\left(1-u^2\right)^{\frac{\ell-3}{2}}\de{u}.\label{e:skew_compl_y}
	\end{align}
	We now add together \eqref{e:skew_compl_x} and \eqref{e:skew_compl_y} to obtain 
	twice the skew-product. When adding the first terms of the above expressions,
	we note that the integrand can be written as $F\left(z\right)
	\left(\frac{\partial}{\partial x}-i\frac{\partial}{\partial y}\right)P\left(z\right)$,
	for some function $F$ and a polynomial $P$. This obviously vanishes since 
	$\left(\frac{\partial}{\partial x}-i\frac{\partial}{\partial y}\right)\left(x+iy\right) = 0$.
	The third term in \eqref{e:skew_compl_y} can be calculated explicitly, while the second terms of
	\eqref{e:skew_compl_x} and \eqref{e:skew_compl_y} can be merged together to obtain
	\begin{align*}
		\left( \pi_{2j},\pi_{2k+1}\right)_{\C}^{\ell} =&
		\frac{1}{2\pi\left(2k+\ell\right)}
		\int\limits_{\mathbb{D}} \overline{z}^{2k}
		\left(1-\overline{z}^2\right)^{\frac{\ell}{2}} 
		z^{2j}	\left(1-z^2\right)^{\frac{\ell}{2}-1}
		\left(i\frac{\partial}{\partial x}-\frac{\partial}{\partial y}\right) \left(
		\sgn\left(\Im z\right)
		\widehat{w}_{\ell}\left(z\right)
		\right) \de{^2 z}
		\\
		&
		- \frac{\ell!}{2^{\ell-1}\Gamma^2\left(\frac{\ell}{2}\right)\left(2k+\ell\right)}	
		B\left(j+k+\frac{1}{2},\ell\right).
	\end{align*}
	The former term coincides with $\left( \pi_{2j},\pi_{2k+1}\right)_{\R}^{\ell}$. For the first one we 
	note
	\begin{equation*}
		\left(i\frac{\partial}{\partial x}-\frac{\partial}{\partial y}\right) \left(
		\sgn\left(\Im z\right)
		\widehat{w}_{\ell}\left(z\right)
		\right) = 2\ell\left(\ell-1\right)
		\frac{\left(1-\left|z\right|^2\right)^{\ell-2}}{\left|1-z^2\right|^{\ell}}
		\left(1-z^2\right).
	\end{equation*}
	This yields
	\begin{align*}
	\left( \pi_{2j},\pi_{2k+1}\right)_{\C}^{\ell} =& 
	-\left( \pi_{2j},\pi_{2k+1}\right)_{\R}^{\ell}
	+ \frac{2\ell\left(\ell-1\right)}{2\pi\left(2k+\ell\right)}
	\int\limits_{\mathbb{D}} \overline{z}^{2k}z^{2j}
	\left(1-\left|z\right|^2\right)^{\ell-2}\de{^2 z}\\
	 =&
	-\left( \pi_{2j},\pi_{2k+1}\right)_{\R}^{\ell} + \frac{1}{\binom{2k+\ell}{\ell}}\delta_{k,j}.
	\end{align*}
	The derivation for $\ell=1$ is even simpler and one should just put 
	$\widehat{w}_{\ell} \left(z\right) \equiv 1$ for all $z$.
\end{proof}

\section{Auxiliary results used in Section \ref{s:Det}}\label{pf:lemma333}
\subsection{Proof of Lemma~\ref{l:Largest_eig} and Lemma~\ref{l:secondtry23}}\label{pf:lemma3}\begin{proof}[Proof of Lemma~\ref{l:Largest_eig}]
	Assume by contradiction that there exists an $n\in\N$ such that $\|H_n\|=1$. 
	Hence, there exists a $\varphi\in \ell^2(\N_0)$ with $\|\varphi\|_2=1$ such that 
	$\langle\varphi, H \varphi\rangle=1$. 
	Since $\|H\|=1$ this implies that $1$ is a boundary point of the numerical range of $H$ 
	but this implies $1$ is a proper eigenvalue \cite{D:57}. 
	This contradicts purely absolutely continuous spectrum of the operator $H$ and the assertion follows. 
\end{proof}

\begin{proof}[Proof of Lemma~\ref{l:secondtry23}]
The operator inequality $H^m 1_{\leq \eps}(H)\leq \eps^{m-1}H$ implies
	\begin{equation*}
	\Tr\big( 1_n H^m 1_{\leq\varepsilon} (H)1_n\big) \leq \varepsilon^{m-1} \Tr\big( 1_n H 1_n\big).
	\end{equation*}
Hence, we obtain
	\begin{align*}
	\sum_{m\in\N} \frac{\Tr\big( 1_n H^m 1_{\leq\varepsilon} (H)1_n\big)}m  
	&\leq \sum_{m\in\N} \frac{\varepsilon^{m-1}} m \Tr\big( 1_n H 1_{\leq\varepsilon} (H)1_n\big)\notag\\
	&\leq
	\frac 1 \varepsilon \log( 1- \varepsilon) \Tr\big( 1_n H h_{\varepsilon}(H)\big),
	\end{align*}
where $h_\eps\in C([0,1])$ is a continuous function such that $1_{<\eps}\leq h_\eps\leq 1_{<2\eps}$. It follows from Lemma \ref{l:moments} that
	\begin{equation*}
	\limsup_{n\to\infty} \frac{\Tr\big( 1_n H^m1_n\big)}{\log n} = \frac{1}{\pi} \int\limits_{0}^{\infty}  \sech^m\left(u \pi\right)\text d u.
	\end{equation*}
From this and a Stone-Weierstra\ss \, argument we infer that
	\begin{equation*}
	\limsup_{n\to\infty} \frac{\Tr\big( 1_n H g(H) 1_n\big)}{\log n} = \frac{1}{\pi} \int\limits_{0}^{\infty}  g(\sech\left(u \pi\right))\sech\left(u \pi\right)\de{u}
	\end{equation*}
for all $g\in C([0,1])$. Therefore, we obtain that
	\begin{equation*}
	\limsup_{n\to\infty}\frac{\Tr\big( 1_n H h_{\varepsilon}(H)\big)}{\log n} = \frac{1}{\pi} \int\limits_{0}^{\infty}  h_\eps(\sech\left(u \pi\right))\sech\left(u 			\pi\right)\text d u
	\end{equation*}
and from dominated convergence the assertion
	\begin{equation*}
	\limsup_{\eps\to 0} \limsup_{n\to\infty}\frac{\Tr\big( 1_n H h_{\varepsilon}(H)\big)}{\log n} = 0.
	\end{equation*}
\end{proof}

\subsection{Proof of Lemma~\ref{l:Pl_asympt} and Lemma~\ref{l:secondtry2}}\label{pf:lemma4}
\begin{proof}[Proof of Lemma~\ref{l:Pl_asympt}]
For $x>0$ and $l\in\N_0$, we define
	\begin{equation*}
	F_l(x):= \Big( \frac 1 4\Big)_l \Big( \frac 1 2\Big)_l \Big( \frac 3 4\Big)_l\textsubscript 4 F_3 \Big(-l, l+\frac 1 2, i \frac x 2, -i 	\frac x 2; \frac 	1 4, \frac 1 2, 		\frac 3 4;1\Big)
	\end{equation*}
Then, it follows from \cite[Eq. (2.5)]{W:91} that we have the asymptotics
 	\begin{equation}\label{e:Pl_asympt}
	F_l(x) = (2 \pi)^{3/2} e^{-3l} l^{3l} \big( 2 | A(i  x/2 )|\cos\big(x\log l + \arg(A(i x/2)\big) 
	+ l^{-1}\widetilde R_l\left(x\right)\big)
	\end{equation}
as $l\to\infty$ with
	\begin{equation*}
	\sup_{x\in[0,M]}\sup_{l\in \N_0} \big|\widetilde R_l\left(x\right)\big|\leq\widetilde r\left(M\right)
	\end{equation*}
for some constant $\widetilde r(M)$ depending on $M$.
We recall
	$| A(i  x/2 )| = \displaystyle\frac  { \cosh(\pi x)^{1/2} }{(2\pi)^{3/2}}.$
On the other hand, the asymptotics 	
	\begin{equation}\label{e:asympt_Ster}
	 \frac{	 e^{-3l} l^{3l}  4^l }{l! \big(\frac 1 2\big)_{2l}} = \frac 1 {2 \sqrt \pi \sqrt l} + O\Big(\frac 1 {l^{3/2}}\Big)
	\end{equation}	
holds as $l\to\infty$. 
To see this, we use that $\big(\frac 1 2\big)_{2l} =\displaystyle \frac{\Gamma(\frac 1 2 + 2l)}{\Gamma(\frac 1 2)}$ and Stirling's formula. This impies
  	\begin{align}\label{e:Stirling}
	 \frac{	e^{-3l} l^{3l}  4^l }{l! \big(\frac 1 2\big)_{2l}} = \frac{ e^{-2l} l^{2l} 4^l  \sqrt \pi}{\sqrt{2\pi l}  \Gamma(2l +\frac 1 2)}\Big(\frac 1 {1+ O(1/l)}\Big).
	\end{align}
Now, the identity $\Gamma(2l +\frac 1 2)= \displaystyle\frac{(4l)! \sqrt \pi}{4^{2l} (2l)! }$ and Stirling's formula give
	\begin{align*}
	\eqref{e:Stirling} &= \frac{e^{-2l} l^{2l} 4^{3l} \sqrt{4 \pi l} (2l/ e)^{2l} }{\sqrt{2 \pi l} \sqrt{8 \pi l} (4l/e)^{4l}}\big( 1 + O(1/l)\big)\notag\\
	&= 
	\frac 1 {2 \sqrt \pi} \frac{n^{4l}}{\sqrt l l^{4l}} \notag\\
	& = \frac 1 {2 \sqrt \pi \sqrt l} \big( 1 + O(1/l)\big).
	\end{align*}
This proves \eqref{e:asympt_Ster}. Inserting  \eqref{e:Pl_asympt} in \eqref{e:defPn}, \eqref{e:asympt_Ster}, gives the assertion. 
\end{proof}

\begin{proof}[Proof of Lemma~\ref{l:secondtry2}]
Using \eqref{e:trace_eq} and Lemma \ref{l:Pl_asympt}, we obtain 
	\begin{align*}
	&\sum_{m\in\N} \frac 1 m  \Tr\big(1_n1_{>\eps}(H) H^m 1_n \big)\notag\\
	=& -\frac 2 \pi \int_0^{\frac{\sech^{-1}(\eps)}\pi}  \log\big( 1-  \sech(x\pi )\big)\notag\\
	&\qquad\qquad \times\Big( \sum_{l=0}^{n-1} \frac 1 l  \cos^2\big( x \log l + \arg(A(i x/2)\big) + l^{-3/2} \hat R_l(x)\Big)\de{x},
	\end{align*}
where  the error term satisfies $\sup_{l\in\N_0} |\hat R_l(x)|< r(\eps)$.
Since the latter is integrable and $l^{-3/2} \hat R_l(x)$ is summable in $l$, we obtain 
	\begin{align}
	&\limsup_{n\to\infty} \frac{\sum_{m\in\N} \frac 1 m  \Tr\big(1_n1_{>\eps}(H) H^m 1_n \big)}{\log n}\notag\\
	=& \limsup_{n\to\infty}-\frac 2 {\pi \log n}  
	\int_0^{\frac{\sech^{-1}(\eps)}\pi}  \log\big( 1-  \sech(x\pi )\big)\notag\\
	&\qquad\qquad\qquad\times 
	\Big( \sum_{l=0}^{n-1} \frac 1 l  \cos^2\big( x \log l + \arg(A(i x/2)\big)\Big)\de{x}.\label{e:lowerbund_eq2}
	\end{align}	
The identity
	\begin{align}
	 \cos^2\big( x \log n + \arg(A(i x/2)\big)
	&= \frac 1 2 
	+\frac{1}{2} \cos\left(2x\log l\right)\cos\left(2\arg A(ix/2)\right)\notag\\
	&\qquad-\frac{1}{2} \sin\left(2x\log l\right)\sin\left(2\arg A(ix/2)\right)\notag\\
	&=: \frac 1 2 + F(x,l)\label{e:lowerbound_eq3}
	\end{align}
and the asymptotics of the harmonic series yield
	\begin{align}
	\eqref{e:lowerbund_eq2} 
	\leq &-\frac 1 \pi \int_0^\infty  \log ( 1- \sech(x \pi ))\de{x} \notag\\
	& +\limsup_{n\to\infty} \frac {-2} {\pi \log n}  \int_0^{\frac{\sinh^{-1}(\eps)}{\pi}}   \log(1-\sech(x \pi ))
	 \sum_{l=1}^{n-1} \frac 1 l F(x,l)\de{x}.\label{e:lowerbound_eq1}
	\end{align}
We are left with estimating the second term. To do this, we split the latter integral further. 
For fixed $\delta>0$ let $0\leq g_\delta\in C_c^\infty\big([0,\frac{\sech^{-1}(\eps)}\pi]\big)$ and $g_\delta(x)=1$ for all $x\in \big(\delta,\sech^{-1}(\eps)-\delta\big)$.
Set 
	\begin{equation*}
	G(x):= \log ( 1- \sech(x \pi ))
	 \sum_{l=1}^{n-1} \frac 1 l F(x,l)
	 \end{equation*}
and we split the latter integral in the following way
	\begin{equation}\label{e:lowerbound_eq4}
		\frac 2 \pi \int_0^{\frac{\sinh^{-1}(\eps)}\pi}   G(x)\de{x}
		 =\frac 2 \pi \int_0^{\frac{\sinh^{-1}(\eps)}\pi}   g_\delta (x)  G(x) \de{x} 
		 + \frac 2 \pi \int_0^{\frac{\sinh^{-1}(\eps)}\pi}  (1- g_\delta(x)) G(x) \de{x} . 
	\end{equation}
The function $ g_\delta( \,\cdot\,) \log ( 1- \sech( \pi \,\cdot\, )) e^{2 i \arg(A(i (\cdot)/2))}$ is 
	smooth, compactly supported and exponentially decaying for any fixed $\eps,\delta>0$. 
	Hence, integration by parts implies that for any $k\in \N$
	\begin{equation}\label{e:lm_eq3}
	\int_0^{\frac{\sinh^{-1}(\eps)}\pi}  g_\delta(x) \log ( 1- \sech(x\pi))  e^{\pm 2i( x \log l + \arg(A(i x/2)))} \de{x}
	 = O(1/(\log l)^k)
	\end{equation}
as $l\to\infty$. 	
	Writing the $\sin$ and $\cos$ terms in $F(x,n)$, see \eqref{e:lowerbound_eq3}, in terms of exponentials and using the latter with $k=2$, we obtain that 
	\begin{equation}\label{e:lowerbound_eq5}
	 \frac 2 \pi \int_0^{\frac{\sinh^{-1}(\eps)}\pi}   g_\delta (x)  G(x) \de{x}= O\Big(\sum_{l=0}^{n-1} \frac 1 {l (\log l)^2} \Big) = O(1),
	\end{equation}
	as $l\to\infty$.
	For the second integral in  \eqref{e:lowerbound_eq4}, we note that 
	\begin{equation*}
	\sup_{x\in (0,\infty)}\sup_{n\in \N_0} |F(n,x)| \leq 1.
	\end{equation*}
	 Therefore, using the asymptotics of the harmonic 		series, we obtain as $n\to\infty$ 
	\begin{equation}\label{e:lowerbound_eq6}
		\Big| \frac 2 \pi \int_0^{\frac{\sinh^{-1}(\eps)}\pi}   \big(1- g_\delta(x)\big) G(x) \de{x} \Big|\leq 
		 \log n \int_0^\infty \big|\big(1- g_\delta(x)\big) \log ( 1- \sech(x\pi))\big| \de{x}  + O(1).
	\end{equation}
We obtain from \eqref{e:lowerbound_eq1}, \eqref{e:lowerbound_eq4},  \eqref{e:lowerbound_eq5} and \eqref{e:lowerbound_eq6} that for all $\delta>0$
 	\begin{align*}
		\limsup_{n\to\infty} \frac{\sum_{m\in\N} \frac 1 m  \tr\big(1_n H^m 1_{>\eps}(H) 1_n \big)}{\log n}
		&\leq 
		-\frac 1 \pi \int_0^\infty \log ( 1- \sech(x \pi ))d x\notag\\
		&\qquad\qquad+ \int_0^{\frac{\sech^{-1}(\eps)}\pi}\big|\big(1- g_\delta(x)\big) \log ( 1- \sech(x\pi))\big|\de{x}.
	\end{align*}   		
Taking the limit $\delta\to 0$, the last term in the latter vanishes by dominated convergence using that $\log ( 1- \sech(\,\cdot \,\pi))$ is integrable. This gives the assertion.	
	\end{proof}

\section{Open problems and conjectures}\label{s:OpenProblems}
In this paper we derived an explicit expression \eqref{e:Pers_det} for the "persistence" probability of truncations  of random orthogonal matrices of size $\ell$. In the case $\ell=1$ we were also able to perform an asymptotic analysis of this probability in \eqref{asymptProbab}. It is natural to ask what happens for $\ell>1$ or even for $\ell$ growing with $n$. These questions have their own applications to the distribution of roots of random, matrix valued, polynomials with coefficients
given by Real Ginibre matrices of size $\ell\times \ell$, see \cite{Fo:10} for details.
Looking at \eqref{e:Pers_det}, one has to analyse the
 determinant of identity minus a weighted Hankel matrix in the case of $\ell>1$. We claim that the methods described in this paper are also
applicable to this case and a similar analysis using the results of \cite{SS:19}, gives
\begin{conjecture}\label{c:large_l}
	Let $\ell$ be a fixed integer number and the ensemble of random matrices $M_{2n}$ be defined as 
	in Theorem~\ref{t:Main_Det}. Then the corresponding persistence probability decays as
	\begin{equation}\label{e:large_l}
	\lim\limits_{n\to\infty}\frac{\log p_{2n}^{\left(\ell\right)}}{\log n} = -2\theta\left(\ell\right),
	\quad\mbox{ with }\quad
	\theta\left(\ell\right) = -\frac{1}{2\pi} \int\limits_{0}^{\infty} 
	\log\Bigg(1-	\Bigg|\frac{\Gamma\left(\frac{\ell}{2}+ix\right)}{
	\Gamma\left(\frac{\ell}{2}\right)}\Bigg|^2\Bigg)\de{x}.
	\end{equation}
\end{conjecture}
A rigorous proof of the above will be the content of a future work. 
So far we were not successful in finding a closed form  of
the integral \eqref{e:large_l}. However, we could rewrite the above in terms 
of a random walk. More precisely, we obtain 
\begin{lemma}
	Let $\left\{\xi_j\right\}_{j=1}^{\infty}$ be a family of i.i.d. random variables 
	having probability density function
	\begin{equation*}
		\rho_{\ell}\left(x\right) = 
		\frac{1}{2B\left(\frac{\ell}{2},\frac{1}{2}\right)}
		\sech^{\ell}\left(\frac{x}{2}\right),
	\end{equation*}
	and $S_k=\sum\limits_{j=1}^k \xi_j$ be a random walk with 
	corresponding steps. Let $\tau$ be the first hitting time of the origin, then
	\begin{equation*}
		\theta({\ell}) = \frac{1}{4}\P{S_{\tau} \in \de{0}},
	\end{equation*}
	where by $\P{\zeta \in \de{0}}$ we mean the probability density function 
	of the random variable $\zeta$ evaluated at the origin.
\end{lemma}
This lemma will also be proved in a future work. 
As discussed in the introduction, the persistence probability of rank-one truncations of 
random orthogonal matrices has evident connections to the persistence probability 
of $\sech$ correlated Gaussian Stationary Processes (GSP). The analysis of the corresponding GSP
led us in \cite{PS:19} to the study of the related persistence problem for the latter random walk
with the parameter $\ell$ set to one. One may expect that for general $\ell>1$ there
should be a connection of \eqref{e:large_l} to GSP with $\sech^\ell\big(\frac{x}{2}\big)$ correlated process.
However, an accurate comparison of our numerical results with the one found in \cite{NL:01} shows some mismatch. 

Another intriguing and challenging question is to study the asymptotics of the persistence problem
for our ensemble of random matrices when the parameter $\ell$ is growing in $n$. Here one 
would expect some phase transition from the weak non-orthogonality universality class, corresponding to 
$\ell/n = o\left(1\right)$, to the Real Ginibre universality class when $\ell/n \to \infty$ 
(compare to Proposition~\ref{p:all_real}). 
In full generality the problem is yet to be solved, but some partial results can be already obtained
given the above conjecture.
\begin{conjecture}\label{p:all_real_large_l}
	For large integers $\ell$ the decay exponent $\theta(\ell)$ behaves as
	\begin{equation}\label{e:theta_large_l}
		\theta \left(\ell\right) = \frac{1}{4}\sqrt{\frac{\ell}{2\pi}}\zeta\left(3/2\right)
		\left(1+o\left(1\right)\right),
		\quad \ell\to\infty.
	\end{equation}
\end{conjecture}
This can be either confirmed by an asymptotic analysis of the Gamma-function or by 
approximating the random walk described above by a random walk with Gaussian $N\left(0,\frac{4}{\ell}\right)$
distributed steps. By formally taking $\ell = 2n$, corresponding to a transition
from singular to non-singular measure in \eqref{e:denMN}, one gets half of the corresponding result
for the Real Ginibre ensemble (see, \cite[Thm. 1.1]{KPTTZ:16}). The factor one half originates from
the fact that truncated orthogonal matrix can not have eigenvalues outside the unit disk, but
the Real Gibre random matrix can.

Apart from studying the probability of having no real eigenvalues for a random matrix, 
one can also look at the probability $p_{2n,2k}^{\left(\ell\right)}$ of having $2k$ real 
eigenvalues. For $k=2n$ we computed this probability in  Proposition~\ref{p:all_real}. 
In the intermediate regime, $0< k < 2n$ we expect that the answer doesn't change until the point when $k$ 
changes from $0$ to roughly the average number of real roots, see similar results \cite{KPTTZ:16}. For $k$ being of order of $n$, 
analogously to the result of \cite{dMPTW:16}, we expect that the probability will decay 
exponentially in terms of $n^2$ with some non-trivial coefficient depending on a ratio $k/n$. 
For other values of $k$ the problem seem very challenging and technical.

\appendix
\section{Volume of orthogonal group}\label{s:v_N}

\begin{proposition}\label{p:delta_m}
	Let $A$ be a real number. Then 
	\begin{equation*}
		I_m\left(A\right) = 
		\int\limits_{\R^{m}} \delta\left(\vec{x}^T\vec{x}-A\right)\de x = 
		\frac{\pi^{\frac{m}{2}} A_+^{\frac{m-2}{2}}}{\Gamma\left(\frac{m}{2}\right)},
		\,\, \mbox{where}\quad x_+ = \max\left\{0,x\right\}.
	\end{equation*}
\end{proposition}
\begin{proof}[Proof of Proposition~\ref{p:delta_m}]
	For $m=1$ the statement is obvious. For $m\geq 2$ we change to polar coordinates
	The integral above is now equal to
	\begin{align*}
		I_m\left(A\right) 
		&= \int\limits_{0}^{\infty}r^{m-1}\delta\left(r^2-A\right)\de{r}
		\int\limits_{0}^{2\pi} \de{\phi_{m-1}}
		\prod\limits_{j=1}^{m-2}\int\limits_{0}^{\pi}\sin^{m-1-j}\phi_{j}\de{\phi_j}
		\\
		&= \frac{1}{2}A_+^{\frac{m-2}{2}}2\pi\prod_{j=1}^{m-2}
		B\left(\frac{m-j}{2},\frac{1}{2}\right)
		= \frac{\pi^{\frac{m}{2}} A_+^{\frac{m-2}{2}}}{\Gamma\left(\frac{m}{2}\right)}.		
	\end{align*}
\end{proof}

\begin{proposition}\label{p:v_N}
	The volume of the orthogonal group $O\left(N\right)$ 
	is equal to
	\begin{equation*}
		v_N = \int\limits_{\R^{N^2}} \delta\left(O^TO-I_N\right)\De{O} = 
		\prod\limits_{j=1}^N \frac{\pi^{j/2}}{\Gamma\left(\frac{j}{2}\right)},
	\end{equation*}
	where $\De{O}$ is the flat Lebesgue measure on $\R^{N^2}$.
\end{proposition}
\begin{remark}
	This is different to what was stated in \cite{KSZ:10}.
\end{remark}
\begin{proof}[Proof of Proposition~\ref{p:v_N}]
	We proof the statement by induction. For $N=1$ one can easily check that
	\begin{equation*}
		\int\limits_{\R}\delta\left(x^2-1\right)\de{x} = 1.
	\end{equation*}
	Let now $N=k+1$ and we split every matrix into blocks of the following form
	\begin{equation*}
		O_{k+1} = 
		\begin{pmatrix}
			m & \vec{b}^T \\
			\vec{c} & D
		\end{pmatrix},
	\end{equation*}
	where $\vec{b},\vec{c}$ are $k$-dimensional column vectors and $D$ is a 
	$k\times k$ real 
	matrix. The corresponding integral can now be written as
	\begin{equation*}
		v_{k+1} =  \int
		\delta\left(m^2+\vec{c}^T\vec{c}-1\right)
		\delta\left(m\vec{b}^T+\vec{c}^TD\right)
		\delta\left(\vec{b}\vec{b}^T+D^TD-I_k\right)
		\de{m}\de{\vec{b}} 
		\de{\vec{c}} \De{D}.
	\end{equation*}
	Integrating out $\vec{b}$, we obtain 
	\begin{equation*}
		v_{k+1} =  \int\de{m}
		\de{\vec{c}} \De{D}
		\delta\left(m^2+\vec{c}^T\vec{c}-1\right)
		\delta\left(D^T\frac{\vec{c}\vec{c}^T}{m^2}D+D^TD-I_k\right)
		\left|m\right|^{-k}.
	\end{equation*}
	Integration over $D$ can be performed by using the induction hypothesis. We define
	\begin{equation*}
		V = I_k + \frac{\vec{c}\vec{c}^T}{m^2},
	\end{equation*}
	which is a real symmetric, positive definite rank one perturbation of the identity
	with determinant $\det V = I + \frac{\vec{c}^T\vec{c}}{m^2}$.
	Changing variables with 
	\begin{equation*}
		D = V^{-1/2}\hat{D},
	\end{equation*}
	one gets
	\begin{align*}
		v_{k+1} &=  \int
		\delta\left(m^2+\vec{c}^T\vec{c}-1\right)
		\delta\left(\hat{D}^T\hat{D}-I_k\right)
		\left|m\right|^{-k}\det^{-k/2}V
		\de{m}
		\de{\vec{c}} \De{\hat{D}}\\
		&= v_k \int	
		\delta\left(m^2+\vec{c}^T\vec{c}-1\right)
		\left(m^2+\vec{c}^T\vec{c}\right)^{-k/2}
		\de{m}
		\de{\vec{c}} \\
		&= v_k \int	
		\delta\left(m^2+\vec{c}^T\vec{c}-1\right)
		\de{m}
		\de{\vec{c}} .
	\end{align*}
	Finally, we integrate over $\vec{c}$ using Proposition~\ref{p:delta_m} and obtain
	\begin{equation*}
		v_{k+1} = v_k \frac{\pi^{\frac{k}{2}}}{\Gamma\left(\frac{k}{2}\right)} 
		\int\limits_{\R} \de{m } \left(1-m^2\right)_+^{\frac{k}{2}-1}
		= v_k \frac{\pi^{\frac{k}{2}}}{\Gamma\left(\frac{k}{2}\right)} 
		B\left(\frac{k}{2},\frac{1}{2}\right) = 
		v_k\frac{\pi^{\frac{k+1}{2}}}{\Gamma\left(\frac{k+1}{2}\right)},
	\end{equation*}
	and the statement follows.
\end{proof}

\section{Properties of Pfaffians and Proof of Proposition~\ref{p:checkboard}}\label{app:Auxiliary}
The Pfaffian is an analogue of the determinant defined for skew-symmetric matrices of even size. 
Let $A = \left\{a_{j,k}\right\}_{j,k=1}^{2n}$ be a skew-symmetric matrix with entries $a_{j,k} = - a_{k,j},
\, j,k=1,\ldots,2n$. Then its Pfaffian is defined by
\begin{equation*}
	\Pf A = \frac{1}{2^nn!}\sum\limits_{\sigma \in S_{2n}} 
	\sgn\left(\sigma\right)\prod\limits_{j=1}^n a_{\sigma\left(2j-1\right)}
	a_{\sigma\left(2j\right)},
\end{equation*}
where the sum is taken over all permutations of elements $\left(1,2,\ldots,2n\right)$. For
skew-symmetric matrices of odd size the Pfaffian is defined to be zero. The Pfaffian can be thought as a square
root of the determinant because of an identity
\begin{equation*}
	\Pf^2 A = \det A,
\end{equation*}
valid vor any skew-symmetric matrix. Below we also use another definition of the Pfaffian via integration 
over Grassmann (anticommuting) variables. Let $\left(\phi_1,\phi_2,\ldots,\phi_j,\ldots\right)$ and
$\left(\psi_1,\psi_2,\ldots,\psi_j,\ldots\right)$ be two families of anticommuting variables
\begin{equation}\label{propertiesGrassmann}
	\phi_j\phi_k=-\phi_k\phi_j, 
	\phi_j\psi_k=-\psi_k\phi_j,
	\psi_j\psi_k=-\psi_k\psi_j.
\end{equation}
Functions of Grassmann variables are defined by the corresponding Taylor series, which are always finite because
of Grassmann variables being nilpotent. The Berezin integral with respect to these 
variables is formally defined using the identities
\begin{equation}\label{propertiesGrassmann2}
	\int \de{\phi_j} = \int \de{\psi_j} = 0, \int \de{\phi_j} \phi_j = \int \de{\psi_j} \psi_j = 1,
\end{equation}
and a multiple integral is defined to be a repeated one. Then for any matrix $M$ of size
$n\times n$ one can see that
\begin{equation*}
	\int \de{\phi_1}\de{\psi_1}\de{\phi_2}\de{\psi_2}\ldots\de{\phi_n}\de{\psi_n}
	\exp{
	-\sum\limits_{j,k=1}^n M_{j,k}\phi_j\psi_k 
	} = \det M.
\end{equation*}
The above follows from two simple observations: Expanding the exponential function and using \eqref{propertiesGrassmann} and \eqref{propertiesGrassmann2}, one sees that the integral on the left is given by the coefficient
in front of the monomial $\phi_1\psi_1\phi_2\psi_2\ldots\phi_n\psi_n$. This term comes only from
expanding $\displaystyle\frac{1}{n!}\Big(-\sum\limits_{j,k=1}^n M_{j,k}\phi_j\psi_k \Big)^n$. Analogously,
one can also write the Pfaffian in terms of the Berezin integral. Let $M$ be a skew-symmetric matrix of size $2n\times 2n$,
then the result reads 
\begin{equation*}
	\int
	\de{\phi_1}\ldots\de{\phi_{2n}}
	\exp{-\dfrac{1}{2}\sum\limits_{j,k=1}^{2n}M_{j,k}\phi_j\phi_k} = \Pf M.
\end{equation*}
This also follows from finding the coefficient in front of the monomial 
$\phi_1\phi_2\ldots\phi_{2n}$ that in its turn comes from expanding 
$\displaystyle\frac{1}{n!}\Big(-\dfrac{1}{2}\sum\limits_{j,k=1}^{2n}M_{j,k}\phi_j\phi_k \Big)^n$.
For more information about Berezin integrals and Grassmann variables we refer to \cite{Be:87}
and about Pfaffians to \cite{Ho:15}.

\begin{proof}[Proof of Proposition~\ref{p:checkboard}]
	Writing $\Pf A$ as a Berezin integral over Grassmann variables, we obtain
	\begin{equation*}
	\Pf A = \int
	\de{\psi_0}\ldots\de{\psi_{2n-1}}
	\exp{-\dfrac{1}{2}\sum\limits_{j,k=0}^{2n-1}
		A_{j,k}\psi_j\psi_k}.			
	\end{equation*}
	$A$ has checkboard pattern, and therefore there are no
	terms in the exponent containing $\psi_j\psi_k$ with even $j-k$.
	Let us split the Grassmann variables into two groups: with even and odd indexes
	which do not "interact". Then
	\begin{align*}
	\Pf A &= \int
	\de{\psi_0}\ldots\de{\psi_{2n-1}}
	\exp{-\dfrac{1}{2}\sum\limits_{j,k=0}^{n-1}\psi_{2j}\psi_{2k-1}
		\left(A_{2j,2k-1}-A_{2k-1,2j}\right)}
	\\
	&=
	\int
	\de{\psi_0}\ldots\de{\psi_{2n-1}}
	\exp{-\sum\limits_{j,k=0}^{n-1}\psi_{2j}\psi_{2k-1}A_{2j,2k-1}}=
	\det A',			
	\end{align*}
	where we used determinant representation via Grassmann variables with $A' = 
	\left\{a_{2i,2j+1}\right\}_{i,j=0}^{n-1}$.
\end{proof}

\end{document}